\documentclass[reqno,10pt,a4paper]{amsart}

\usepackage{subfigure}
\usepackage{amsmath,amssymb,amsthm,bbm,marvosym}
\usepackage{epsfig,psfrag,color}
\usepackage{enumitem}
\usepackage[]{units}
\setenumerate{leftmargin=*}

\let\subset\subseteq 
\let\eps\varepsilon
\let\rho\varrho
\def\dcup{\dot\cup}
\def\NN{\mathbb{N}}
\def\RR{\mathbb{R}}
\def\tB{t_{\mathrm B}}
\def\BOX{\mathrm{Box}}
\def\esssup{\mathrm{ess\,sup\,}}
\def\essinf{\mathrm{ess\,inf\,}}

\newcommand{\JUSTIFY}[1]{\fbox{\tiny{#1}}\quad}

\def\differential{\mathsf{d}} 

\def\Pr{\mathbf{P}} 
\def\Exp{\mathbf{E}} 
\def\Var{\mathbf{Var}} 
\def\RG{\mathbb{G}} 
\def\RH{\mathbb{H}} 
\def\BRG{\mathbb{B}} 

\newtheorem{theorem}{Theorem}[section]
\newtheorem{lemma}[theorem] {Lemma}   
\newtheorem{corollary}[theorem] {Corollary}   
  
\newtheorem{proposition}[theorem] {Proposition}  
\newtheorem{fact}[theorem]{Fact}

\newtheorem*{dreamlemma*}{Dream Lemma}   
\newtheorem{claim}{Claim}[theorem]
\theoremstyle{definition}
\newtheorem{remark}[theorem] {Remark} 
\newtheorem*{remark*} {Remark} 
\newtheorem{example}[theorem] {Example}
\newtheorem{definition}[theorem] {Definition}

\newcommand{\By}[2]{\overset{\mbox{\tiny{#1}}}{#2}}
\newcommand{\ByRef}[2]{   \By{\eqref{#1}}{#2} }

\newcommand{\geBy}[1]{    \By{#1}{\ge} }
\newcommand{\eqByRef}[1]{ \ByRef{#1}{=} }

\newcommand{\gByRef}[1]{  \ByRef{#1}{>} }
\newcommand{\leByRef}[1]{ \ByRef{#1}{\le} }
\newcommand{\geByRef}[1]{ \ByRef{#1}{\ge} }

\renewcommand{\leq}{\leqslant}
\renewcommand{\le}{\leqslant}
\renewcommand{\geq}{\geqslant}
\renewcommand{\ge}{\geqslant}

\renewcommand{\epsilon}{\varepsilon}

\let\oldmarginpar\marginpar
\renewcommand\marginpar[1]{\-\oldmarginpar[\raggedleft\footnotesize #1]%
{\raggedright\footnotesize #1}}

\newcommand{\Referee}[1]{}

\title{Cliques in dense inhomogeneous random graphs}
\author{Martin Dole\v zal}
\address{Institute of Mathematics, Czech Academy of Sciences. \v Zitn\'a 25, 110 00, Praha, Czech Republic. The Institute of Mathematics of the Czech Academy of Sciences is supported by RVO:67985840.}
\email{dolezal@math.cas.cz}
\author{Jan Hladk\'y}
\address{The Czech Academy of Sciences, Institute of Computer Science, Pod Vod\'{a}renskou v\v{e}\v{z}\'{\i} 2, 182 07 Prague, Czech Republic. With institutional support RVO:67985807.
\emph{Address at time of submission:} Institute of Mathematics, Czech Academy of Sciences. \v Zitn\'a 25, 110 00, Praha, Czech Republic. The Institute of Mathematics of the Czech Academy of Sciences is supported by RVO:67985840.}
\email{honzahladky@gmail.com}
\thanks{MD: Supported by GA \v CR Grant 16-07378S and  RVO: 67985840}
\thanks{JH: The research leading to these results has received funding from the People Programme (Marie Curie Actions) of the European Union's Seventh Framework Programme (FP7/2007-2013) under REA grant agreement umber 628974}
\author{Andr\'as M\'ath\'e}
\address{Mathematics Institute,
       University of Warwick,
       Coventry, CV4~7AL, UK}
\email{A.Mathe@warwick.ac.uk}   
\thanks{AM was supported by a Leverhulme Trust Early Career Fellowship and by the Hungarian National Research, Development and Innovation Office -- NKFIH, 104178.}
\date{}

\begin{document}
\begin{abstract}
The theory of dense graph limits comes with a natural sampling process which yields an inhomogeneous variant $\RG(n,W)$ of the Erd\H{o}s--R\'{e}nyi random graph. Here we study the clique number of these random graphs. We establish the concentration of the clique number of $\RG(n,W)$ for each fixed $n$, and give examples of graphons for which $\RG(n,W)$ exhibits wild long-term behavior. Our main result is an asymptotic formula which gives the almost sure clique number of these random graphs. We obtain a similar result for the bipartite version of the problem. We also make an observation that might be of independent interest: Every graphon avoiding a fixed graph is countably-partite.
\end{abstract}
\maketitle\Referee{5}

\section{Introduction}
\label{sec:Intro}
The Erd\H{o}s--R\'enyi random graph $\RG(n,p)$ is a random graph with vertex set $[n]=\{1,\ldots,n\}$, where each edge is included independently with probability $p$. Since Gilbert, and independently Erd\H{o}s and R\'enyi introduced the model in~1959, this has been arguably the most studied random discrete structure. Here, we recall facts about cliques in $\RG(n,p)$; these were among the first properties studied in the model. The key question in the area concerns the order of the largest clique. We write $\omega(G)$ for the order of the largest clique in a graph $G$.
Matula~\cite{Matula:LargestClique}, and independently Grimmett and McDiarmid~\cite{GriMcD:ColouringRandom} have shown that when $p\in (0,1)$ is fixed, and $\epsilon>0$ is arbitrary, we have
\begin{equation}\label{eq:cliqueGnp}
\lim_{n\rightarrow\infty}\Pr\left[\frac{\omega(\RG(n,p))}{\log n}=(1\pm\epsilon)\frac{2}{\log(1/p)}\right]=1\;.
\end{equation}
Here, as well as in the rest of the paper, we use $\log$ for the natural logarithm.
The actual result is stronger in two directions: firstly, it extends also to sequences of probabilities $p_n\rightarrow 0$, and secondly, Matula, Grimmett and McDiarmid proved an asymptotic concentration of $\omega(\RG(n,p))$ on two consecutive values for which they provided an explicit formula. 

Our aim however was extending~\eqref{eq:cliqueGnp} in a different direction. That direction is motivated by the theory of limits of dense graph sequences. Let us introduce the key concepts of the theory; for a thorough treatise we refer to Lov\'asz's book~\cite{Lovasz:LimitBook}. The key object of that theory are graphons. A \emph{graphon} is a symmetric measurable function $W:\Omega^2\rightarrow [0,1]$, where $\Omega$ is a probability space. In their foundational work, Lov\'asz and Szegedy~\cite{LovSze07} proved that each sequence of finite graphs contains a subsequence that converges --- in the so-called \emph{cut metric} --- to a graphon. This itself does not justify the graphons as limit objects; it still could be that the space of graphons is unnecessarily big. In other words, one would like to know that every graphon $W$ is attained as a limit of finite graphs. To this end, Lov\'asz and Szegedy introduced the random graph model $\RG(n,W)$. The set of vertices of $G\sim \RG(n,W)$ is the set $[n]$. To sample $G$, first generate $n$ independent elements $x_1,\ldots,x_n\in \Omega$ according to the law of $\Omega$. Then, connect $i$ and $j$ by an edge with probability $W(x_i,x_j)$ (independently of other choices). Lov\'asz and Szegedy proved that with probability one, the sequence of samples from $\RG(n,W)$ converges to~$W$.

The strength of the theory of graph limits is that convergence in the cut metric implies convergence of many key statistics of graphs (or graphons). These include frequencies of appearances of small subgraphs, and normalized MAX-CUT-like properties. An important direction of research is to understand which other parameters are continuous with respect to the cut metric; those parameters can then be defined even on graphons. In Example~\ref{ex:notcontinuous} below we show that the clique number can be very discontinuous with respect to the cut metric. Not all is lost even in such a situation. One can still study discontinuous parameters of a graphon~$W$ via the sampling procedure~$\RG(n,W)$. The random sampling will suppress pathological counterexamples such as those in Example~\ref{ex:notcontinuous}, and thus allow us to associate limit information even to these limit parameters. 

\begin{example}\label{ex:notcontinuous}
Let $f:\NN\rightarrow \NN$ be a function that tends to infinity very slowly. 
\begin{itemize}
\item Let us consider a sequence of $n$-vertex graphs consisting of a clique of order $\lceil n/f(n)\rceil$ and $n-\lceil n/f(n)\rceil$ isolated vertices. This sequence converges to the 0-graphon, the smallest graphon in the graphon space. Yet, the clique numbers are unusually high, almost of order $\Theta(n)$.
\item Let us consider a sequence of $n$-vertex $f(n)$-partite Tur\'an graphs. This sequence converges to the 1-graphon, the largest graphon in the graphon space. Yet, the clique numbers tend to infinity very slowly.\qed
\end{itemize}
\end{example}

By the previous example, we see that
\begin{itemize}
	\item there are sequences of finite graphs with clique numbers growing much faster than logarithmic with a limit graphon $W\equiv 0$,
	\item while there are other sequences of finite graphs with clique numbers growing much slower than logarithmic with a limit graphon $W\equiv 1$.
\end{itemize}

So, suppose that we have suppressed these pathological examples by looking at ``typical graphs that are close to $W$'' rather than ``all graphs that are close to $W$'', and let us see what value motivated by the clique number can be associated to $W$. To this end, suppose that $W\colon \Omega^2\rightarrow [0,1]$ is such a graphon that $W(x,y)\in [p_1,p_2]$ for every $x,y\in \Omega$, where $0<p_1\le p_2<1$ are fixed. Then the edges of $\RG(n,W)$ are stochastically between $\RG(n,p_1)$ and $\RG(n,p_2)$. Thus,~\eqref{eq:cliqueGnp} tells us that the clique number $\omega(\RG(n,W))$ asymptotically almost surely satisfies
\begin{equation}\label{eq:sandwichp1p2}
(1-o(1))\frac{2}{\log(1/p_1)}\le\frac{\omega(\RG(n,W))}{\log n}\le (1+o(1))\frac{2}{\log(1/p_2)}\;.
\end{equation}
Thus, it is actually plausible to believe that 
\begin{equation}\label{eq:mainquantity}
\frac{\omega(\RG(n,W))}{\log n}
\end{equation} converges in probability. In this paper, we study  this and related questions. 

\subsection{Related literature on inhomogenous random graphs}\label{ssec:relatedliterature}\Referee{4}
Inhomogeneous random graphs allow one to express different intensities of bonds between the corresponding parts of the base space. This is obviously useful in modeling phenomena in biology, sociology, computer science, physics, and other settings. The price one has to pay for this flexibility  of these models is in extra difficulties in  mathematical analysis of their properties. This is one of the reasons why literature on $\RG(n,W)$ is fairly scarce, compared to $\RG(n,p)$; another reason apparently being that the inhomogeneous model is much more recent. Actually, in the inhomogeneous model, most work was done in the \emph{sparse regime}, which we shall introduce now. To get a counterpart to sparse random graphs $\RG(n,p_n)$, $p_n\to 0$, one introduces rescaling $\RG(n,p_n\cdot W)$. In this setting, $W$ need not be bounded from above anymore (even though the question of ``how unbounded'' $W$ can be is rather subtle and we neglect it here). 

The most impressive example of work concerning sparse inhomogeneous random graphs is \cite{BolJanRio:PhaseTransition} in which the existence and the size of the giant component in $\RG(n,\frac1n\cdot W)$ was determined. This work has initiated a big amount of further work on $\RG(n,\frac1n\cdot W)$, such as~\cite{Tur:LargestSubcritical,vanderHofstad:Critical}, as well as on related percolation models~\cite{BolBorChaRio:PercolationDense}.

The threshold for connectivity of inhomogeneous random graphs was investigated in~\cite{DevFra}. The diameter of inhomogeneous random graphs was studied in~\cite{FrMi:Diameter}.

A particular subclass of the random graph models $\RG(n,W)$ are the so-called \emph{stochastic block models} introduced already in 1980's in the field of mathematical sociology~\cite{HolLasLei:StochasticBlockmodels}. They are used extensively in many areas of mathematics, computer science, and physics. In our language, (the dense version of) stochastic block models correspond to the case when $W$ is a step-function with finitely many or countably many steps. The stochastic block model is mathematically much more tractable. For example, the study of criticality in stochastic block models in~\cite{KaKoPa:PhaseTransitionBlock} seems to be much more tractable than in the case of general inhomogeneous random graphs.

\section{Our contribution}
In this section we present our main results. The notation in this section is standard. We refer the reader to Section~\ref{sec:preliminaries} for formal definitions.

We saw that for many natural graphons $W$, $\omega(\RG(n,W))$ grows logarithmically. It is easy to construct a graphon for which $\omega(\RG(n,W))$ grows for example as $\log\log n$, or another graphon for which $\omega(\RG(n,W))$ grows for example as $n^{0.99}$. More surprisingly, our next proposition shows that we can have an oscillation between these two regimes even for one graphon.
\begin{proposition}\label{prop:wild}
For an arbitrary function $f:\NN\rightarrow\RR_+$ with $\lim_{n\to\infty}f(n)=+\infty$ there exists a graphon $W$ and a sequence of integers $1=\ell_0<k_1<\ell_1<k_2<\ell_2<\ldots$ such that asymptotically almost surely,
\begin{align}
\label{eq:Psismall}\omega(\RG(k_i,W))&<f(k_i)\;\mbox{, and}\\
\label{eq:Psilarge}\omega(\RG(\ell_i,W))&>\frac{\ell_i}{f(\ell_i)}\;.
\end{align}
\end{proposition}
While Proposition~\ref{prop:wild} shows that the long-term behavior of $\omega(\RG(n,W))$ can be quite wild, for a fixed (but large) $n$, the distribution of $\omega(\RG(n,W))$ is concentrated. 
\begin{theorem}\label{thm:concentration}
For each graphon $W$ and each $n$, we have that for each $\epsilon>0$, 
$$
\lim_{n\to \infty}
\Pr\left[\:\left|\frac{\omega(\RG(n,W))}{\Exp[\omega(\RG(n,W))]}-1\right|>\epsilon\:\right]=0
\;.$$
\end{theorem}
The proofs of Proposition~\ref{prop:wild} and Theorem~\ref{thm:concentration} are given in Section~\ref{sec:examples}. In the proof of Theorem~\ref{thm:concentration}, we need to consider the case of graphons $W$ for which $\Exp[\omega(\RG(n,W))]$ is bounded (as $n\rightarrow\infty$) separately. Investigation of such graphons led to a result that is of independent interest. Let us give the details.
Clearly, for each graphon $W$, and each $n\in\NN$, $\omega(\RG(n,W))$ is stochastically dominated by $\omega(\RG(n+1,W))$. As a consequence, the sequence $\Exp[\omega(\RG(1,W))], \Exp[\omega(\RG(2,W))], \Exp[\omega(\RG(3,W))], \ldots$ is nondecreasing. We say that $W$ \emph{has a bounded clique number} if $\lim_{n\to\infty} \Exp[\omega(\RG(n,W))]<+\infty$. Note that one example of graphons of bounded clique numbers are graphons $W$ which have zero homomorphism density of $H$ (see~\eqref{eq:defhomden} for the definition) for some finite graph~$H$.
A subclass of these are \emph{$k$-partite
graphons}. These are graphons $W:\Omega^2\to [0,1]$ for which there exists a
measurable partition $\Omega=\Omega_1\dcup \Omega_2\dcup \ldots \dcup \Omega_k$ such that for each
$i\in[k]$, $W\restriction_{\Omega_i\times \Omega_i}=0$ almost everywhere. In the following example,
we show that the structure of graphons with a bounded clique number can be more
complicated. We consider a sequence of triangle-free graphs $G_1,G_2,\ldots$
whose chromatic numbers tend to infinity (it is a standard exercise that such
graphs indeed exist). Let $W_1,W_2,\ldots$ be their graphon representations. We now glue these graphons into one graphon $W$. Clearly,
$\omega(\RG(n,W))\le 2$ with probability one, but $W$ is not $k$-partite for any $k$.
Here, we show that the structure of graphons with a bounded clique number cannot be
much more complicated than in the example above.
We call a graphon $W:\Omega^2\to [0,1]$ \emph{countably-partite}, if there
exists a measurable partition $\Omega=\Omega_1\dcup \Omega_2\dcup \ldots$ such that
for each $i\in\mathbb N$, $W\restriction_{\Omega_i\times \Omega_i}=0$ almost everywhere. 
\begin{theorem}\label{thm:boundedcliquenumber}
Every graphon with a bounded clique number is countably-partite.
\end{theorem}
The proof of Theorem~\ref{thm:boundedcliquenumber} is given in Section~\ref{ssec:boundedclique}.

\medskip
Let us turn our attention to the main subject of the paper, that is, to the behavior of the clique number in $\mathbb G(n,W)$ scaled as in~\eqref{eq:mainquantity}. As a warm-up for studying~\eqref{eq:mainquantity}, we first deal with its bipartite counterpart. To this end, we shall work with \emph{bigraphons}. Bigraphons, introduced first in~\cite{LoSz:RegularityTopology}, arise as limits of balanced bipartite graphs. A bigraphon is a measurable function $U:\Omega_1\times \Omega_2\rightarrow [0,1]$. Here, $\Omega_1$ and $\Omega_2$ are probability spaces which represent the two partition classes, and the value $U(x,y)$ represents the edge intensity between the parts corresponding to $x$ and $y$. This suggests the sampling procedure for generating the inhomogeneous random bipartite graph $\BRG(n,U)$: we sample uniformly and independently at random points $x_1,\ldots,x_n$ from $\Omega_1$ and $y_1,\ldots,y_n$ from $\Omega_2$. In the bipartite graph $\BRG(n,U)$ with colour classes $\{a_i\}_{i=1}^n$ and $\{b_j\}_{j=1}^n$, we connect $a_i$ with $b_j$ with probability $U(x_i,y_j)$.
We define the natural bipartite counterpart to the clique number. Given a bipartite graph $G=(A,B;E)$, we define its \emph{biclique number} as the largest $\ell$ such that there exist sets $X\subset A$, $Y\subset B$, $|X|=|Y|=\ell$, that induce a complete bipartite graph. We denote the biclique number of $G$ by $\omega_2(G)$.

\medskip

The main result concerning the biclique number is the following.
\begin{theorem}\label{thm:biclique}
Let $U:\Omega_1\times\Omega_2\rightarrow [0,1]$ be a bigraphon whose essential supremum $p=\esssup U$ is strictly between zero and one. Then we asymptotically almost surely have 
$$\omega_2(\BRG(n,U))=(1\pm o(1))\cdot \frac{2}{\log 1/p}\cdot\log n \;.$$
\end{theorem}
The proof of Theorem~\ref{thm:biclique} is given in Section~\ref{sec:bislique}.

\medskip
We turn to our main result which determines the quantity~\eqref{eq:mainquantity}. Suppose that $W$ is a graphon with strictly positive essential infimum. Define
\begin{equation}\label{eq:defKappaSimple}
\kappa(W)=\sup\left\{\frac{2\|h\|_1^2}{\int_{(x,y)\in\Omega^2} h(x)h(y)\log \left(\nicefrac{1}{W(x,y)}\right) \differential(\nu^2)} \::\:h\text{ is a nonnegative } L^1 \text{-function on }\Omega\right\}\;.
\end{equation}
Here, we set $\tfrac 00=0$ and $\tfrac a0=+\infty$ for $a\in\RR\setminus\{0\}$.

We can now state our main result.
\begin{theorem}\label{thm:main}
Suppose that $W$ is a graphon whose essential infimum is strictly positive. Then 
\begin{itemize}
	\item if $\kappa(W)<+\infty$ then a.a.s. $\omega(\RG(n,W))=(1+o(1))\cdot\kappa(W)\cdot\log n$, and
	\item if $\kappa(W)=+\infty$ then
	a.a.s. $\omega(\RG(n,W))\gg \log n$.
\end{itemize}
\end{theorem}
Theorem~\ref{thm:main} is consistent with~\eqref{eq:cliqueGnp}.\Referee{1} Indeed, suppose that $W\equiv p\in(0,1)$ is a constant graphon. Then for any $h$ in~\eqref{eq:defKappaSimple} (which is not constant zero), we have that $\frac{2\|h\|_1^2}{\int_{(x,y)\in\Omega^2} h(x)h(y)\log \left(\nicefrac{1}{W(x,y)}\right)\differential(\nu^2)}=\frac{2}{\log 1/p}$. We provide heuristics for Theorem~\ref{thm:main} for more complicated graphons in Section~\ref{sec:formulaheuristic}. Unfortunately, we were unable to turn these relatively natural heuristics into a rigorous proof. The actual proof of Theorem~\ref{thm:main} is given in Section~\ref{sec:mainproof}, building on tools from Section~\ref{sec:tools}.

There are several alternative ways of expressing $\kappa(W)$. For example,  when we heuristically derive Theorem~\ref{thm:main}, we make use of the following identity.
\begin{fact}\label{fact:kappaalt}
We have
\begin{equation}\label{eq:defKappa}
\kappa(W)=\sup\left\{\|f\|_1 \::\: f\text{ is a nonnegative } L^1 \text{-function on }\Omega, \;\Gamma(f,W)\ge 0\right\}\;,
\end{equation}
where
\begin{equation}\label{eq:A}
\Gamma(f,W)=\underbrace{\int_{x\in\Omega} f(x)\differential(\nu)}_{(*)}+\underbrace{\frac12\int_{(x,y)\in\Omega^2} f(x)f(y)\log W(x,y)\differential(\nu^2)}_{(**)}
\;.\end{equation}
\end{fact}
Other expressions of $\kappa(W)$ are given in Proposition~\ref{proposition:differentkappa}.

\begin{remark}\label{rem:CliqueIndep}\Referee{6}
While we formulate all the problems in terms of cliques, we could have worked with the complementary notion of independent sets instead. Indeed, investigating one of these notions with respect to $\RG(n,W)$ is equivalent to investigating the other with respect to $\RG(n,1-W)$.
\end{remark}
Using this observation, we get from Theorem~\ref{thm:main} the following corollary for the size of the maximum independent set $\alpha(\RG(n,W))$.
\begin{corollary}
Suppose that $W$ is a graphon whose essential supremum is strictly less than~1. Then 
\begin{itemize}
	\item if $\kappa(1-W)<+\infty$ then a.a.s. $\alpha(\RG(n,W))=(1+o(1))\cdot\kappa(1-W)\cdot\log n$, and
	\item if $\kappa(1-W)=+\infty$ then
	a.a.s. $\alpha(\RG(n,W))\gg \log n$.
\end{itemize}
\end{corollary}

\medskip
In our language, (the dense version of) stochastic block models correspond to the case when $W$ is a step-function with finitely many or countably many steps. There exists a conceptually simpler proof of Theorem~\ref{thm:main} when restricted to stochastic block models. This simplification occurs both on the real-analytic side (i.e., general measurable functions versus step-functions) and on the combinatorial side. See our remark in Section~\ref{ssec:failure}.

\section{Preliminaries}\label{sec:preliminaries}
\subsection{Notation} For $n\in\NN$, we write $[n]=\{1,\ldots,n\}$, and $[n]_0=\{0,1,\ldots,n\}$.
As always, we denote by $n\choose k$ the binomial coefficient $\frac{n!}{k!(n-k)!}$. For the multinomial coefficients of higher orders, we employ the notation ${n\choose {k_1 | k_2 | \ldots | k_l}}=\frac{n!}{k_1!k_2!\ldots k_l!(n-\sum k_i)!}$ (here we suppose $k_1+k_2+\ldots +k_l\le n$). We omit rounding symbols where it does not affect correctness of the calculations.

We shall always assume that $\Omega$ is a standard Borel probability space without atoms. We always write $\nu$ for the probability measure associated with $\Omega$. 

We shall sometimes make use of tools from real analysis which are available only for $\RR$ and $\RR^d$. The Lebesgue measure will be denoted by $\lambda$. It should be always clear from the context whether we mean the one-dimensional Lebesgue measure on $\RR$, two dimensional Lebesgue measure on $\RR^2$ or any higher dimensional Lebesgue measure.

We write $\| \cdot \|_1$ to denote the $L^1$-norm of functions (or vectors in a finite-dimensional space). Non-negative vectors, and non-negative $L^1$-functions\footnote{the vector-/function-space will be clear from the context} are called \emph{histograms} (see our explanation at the end of Section~\ref{subsection:FirstMoment}). For a histogram $f$, we write $\BOX(f)$ for all histograms $g$ for which $g\le f$ (pointwise). We say that a histogram is \emph{non-trivial} if it is not almost everywhere zero.

We recall the notions of \emph{essential supremum} and \emph{essential infimum}. Suppose that $\Omega$ is a space equipped with a measure $\nu$. For a measurable function $f\colon X\rightarrow \RR$ we define $\esssup f$ as the least number $a$ such that $\nu(\{x\in \Omega:f(x)> a\})=0$. The quantity $\essinf f$ is defined analogously.

\subsection{Random graphs $\RH(n,W)$}
There is a natural intermediate step when obtaining the random graph $\RG(n,W)$ from a graphon $W$ which is often denoted by $\RH(n,W)$. To obtain $\RH(n,W)$ we sample $n$ random independent points $x_1,\ldots,x_n$ from the probability space underlying $W$. The random graph has the vertex set $[n]$. The edge-set is an edge-set of a complete graph equipped with edge-weights. The weight of the edge $ij$ is $W(x_i,x_j)$. Self-loops are not included.

\subsection{Graphons}
The above crash course in graph limits almost suffices for the purposes of this paper, and we need only a handful of additional standard definitions. See~\cite{Lovasz:LimitBook} for further references.

All (non-discrete) probability spaces in this paper are standard Borel probability spaces without atoms. Recall that the Isomorphism Theorem (see e.g.\ \cite[Theorem 17.41]{Kechris}) tells us that there is a measure-preserving isomorphism between each two such spaces (i.e.\ a bijection between the spaces such that this function and its inverse are measurable and preserve measures). In particular, suppose that $W:\Omega^2\rightarrow [0,1]$ is a graphon defined on a probability space $\Omega$, and let $X$ be another probability space. Let us fix an isomorphism $\psi:X\rightarrow \Omega$ between $X$ and $\Omega$. By a \emph{representation of $W$ on $X$} we mean the graphon $W':X^2\rightarrow [0,1]$, $(x,y) \mapsto W(\psi(x),\psi(y))$. Of course, the representation depends on the actual choice of the isomorphism $\psi$. Note however that the distribution of $\RG(n,W')$ does not depend on the choice of $\psi$ as it is the same as the distribution of $\RG(n,W)$. Note also that from the Isomorphism Theorem above we get the following fact.
\begin{fact}\label{fact:representinterval}
Each graphon can be represented on the open unit interval $(0,1)$.
\end{fact}

We shall need to ``zoom in'' on a certain part of a graphon. The next definition is used to this end.
 \begin{definition}\label{def:subgraphon}
 	Suppose that $W:\Omega^2\rightarrow[0,1]$ is a graphon on a probability space $\Omega$ with a measure $\nu$. By a \emph{subgraphon of $W$ obtained by restricting to a set $A\subset \Omega$} of positive measure we mean a function $U:A^2\rightarrow [0,1]$ which is simply the restriction $W\restriction_{A\times A}$. When working with this notion, we need to turn $A$ to a probability space. That is, we view $U$ as a graphon on the probability space $A$ endowed with measure $\nu_A(B):=\frac{\nu(B)}{\nu(A)}$ for every measurable set $B\subset A$.
 \end{definition}
Observe that in the above setting for every $B\subset A$ of positive measure we have
\begin{equation}\label{eq:rescale}
\frac{1}{\nu(B)^2}\int_{(x,y)\in B\times B}\log(\nicefrac{1}{W(x,y)}) \differential(\nu^2)
=
\frac{1}{\nu_A(B)^2}\int_{(x,y)\in B\times B}\log(\nicefrac{1}{W(x,y)}) \differential(\nu^2_A)\;.
\end{equation}
 
Note that a lower bound on $\omega(\RG(n,W\restriction_{A\times A}))$ provides readily a lower bound on $\omega(\RG(n,W))$. More precisely, suppose that we can show that asymptotically almost surely, $\omega(\RG(n,W\restriction_{A\times A}))\ge c\log n$. Consider now sampling the random graph $\RG(n,W)$. By the Law of Large Numbers, out of the $n$ sampled points $x_1,\ldots,x_n\in \Omega$, there will be $(\nu(A)-o(1))n>\tfrac 12\nu(A)n$ many of them contained in $A$. In other words, there is a coupling of $G=\RG(n,W)$ and $G'=\RG(\tfrac 12\nu(A) n,W\restriction_{A\times A})$ such that with high probability, $G'$ is contained in $G$ as a subgraph. We conclude that 
\begin{equation}\label{eq:PsiSubgraphon}
\omega(\RG(n,W))\geBy{a.a.s.} \omega\left(\RG(\tfrac 12\nu(A) n,W\restriction_{A\times A})\right)
\geBy{a.a.s.} c\log\big(\tfrac 12\nu(A)n\big)=(c-o(1))\log n\;.
\end{equation}

\medskip

The \emph{homomorphism density} of a graph $H=(\{v_1,\ldots,v_\ell\},E)$ in a graphon $W$ is defined by
\begin{equation}\label{eq:defhomden}
t(H,W)=\int_{(x_1,\ldots,x_\ell)\in\Omega^\ell}\;
\prod_{i<j:v_iv_j\in E}W(x_i,x_j)\;\differential(\nu^\ell)\;.
\end{equation}

Suppose that $\Omega$ is an atomless standard Borel probability space. Let $W_1,W_2:\Omega^2\rightarrow[0,1]$ be two graphons. We then define the \emph{cut-norm distance} of $W_1$ and $W_2$ by
$$d_\square(W_1,W_2)=\sup_{S,T}\left|\int_{(x,y)\in S\times T}\left(W_1(x,y)-W_2(x,y)\right) \;\differential(\nu^2)\right|,$$
where $S$ and $T$ range over all measurable subsets of $\Omega$. Strictly speaking, $d_\square$ is only a pseudometric since two graphons differing on a set of measure zero have zero distance. Based on the cut-norm distance we can define the key notion of \emph{cut distance} by
\begin{equation}\label{eq:defcutdist}
\delta_\square(W_1,W_2)=\inf_{\varphi} d_\square(W_1^\varphi,W_2)\;,
\end{equation} where $\varphi:\Omega\rightarrow \Omega$ ranges through all measure preserving automorphisms of $\Omega$, and $W_1^\varphi$ stands for a graphon defined by $W_1^\varphi(x,y)=W_1(\varphi(x),\varphi(y))$.
Then $\delta_\square$ is also a pseudometric.

Suppose that $H=(\{v_1,\ldots,v_\ell\},E)$ is a graph (which is allowed to have self-loops), and let $\Omega$ be an arbitrary atomless standard Borel probability space. By a \emph{graphon representation $W_H$ of $H$} we mean the following construction. We consider an arbitrary partition $\Omega=A_1\dcup A_2\dcup\ldots\dcup A_\ell$ into sets of measure $\frac1\ell$ each. We then define the graphon $W_H$ as $1$ or $0$ on each square $A_i\times A_j$, depending on whether $v_iv_j$ forms an edge or not. Note that $W_H$ is not unique since it depends on the choice of the sets $A_1,\ldots,A_\ell$. However, all the possible graphons $W_H$ are at zero distance in the $\delta_\square$-pseudometric. So, when writing $W_H$ we refer to any representative of the above class. 
With this in mind, we can also define the cut distance of $H$ and any graphon $W\colon\Omega^2\rightarrow [0,1]$, denoted by $\delta_\square(H,W)$, as $\delta_\square(W_H,W)$.
Also, all of this extends in a straightforward way to weighted graphs with a weight function $w:E\rightarrow[0,1]$.
\begin{remark}\label{rem:graphononinterval}
Suppose that $H$ is a finite graph and $I$ is the unit interval (open or closed). Then in the above construction we can take the sets $A_i\subset I$ to be intervals.
\end{remark}

\medskip
The key property of the sampling procedures described earlier (both $\RG(n,W)$ and $\RH(n,W)$) is that the samples are typically very close to the original graphon $W$ in the cut distance. Let us state this fact (for the sampling procedure $\RH(n,W)$), proven first in~\cite{Borgs2008c}, formally.
\begin{lemma}[{\cite[Lemma~10.16]{Lovasz:LimitBook}}]\label{lem:secondsamplinglemma}
Let $W$ be an arbitrary graphon. Then for each $k\in\NN$ with probability at least $1-\exp(-\frac{k}{2\log k})$, we have
$\delta_\square\big(\RH(k,W),W\big)\le\frac{20}{\sqrt{\log k}}$.\qed
\end{lemma}
In some situations, we shall consider a wider class of kernels, the so-called $L^\infty$-graphons. These are just symmetric $L^\infty$-functions $W:\Omega^2\rightarrow \mathbb R_+$. That is, we do not require $L^\infty$-graphons to be bounded by~1 but rather by an arbitrary constant. The random graph $\RH(n,W)$ makes sense even for $L^\infty$-graphons.\footnote{But $\RG(n,W)$ does not make sense.} By simple rescaling, we get the following from Lemma~\ref{lem:secondsamplinglemma}.\Referee{12}
\begin{corollary}\label{cor:secondsamplingunbounded}
Let $W$ an arbitrary $L^\infty$-graphon. Then for each $k\in\NN$ with probability at least $1-\exp(-\frac{k}{2\log k})$, we have $\delta_\square\big(\RH(k,W),W\big)\le\frac{20\|W\|_\infty}{\sqrt{\log k}}$.
\end{corollary}

Let us note that the proof of Lemma~\ref{lem:secondsamplinglemma} is quite involved.

\subsection{Lebesgue points and approximate continuity}
Let $f\colon(0,1)^2\rightarrow\RR$ be an integrable function. Recall that $(x,y)\in(0,1)^2$ is a called a \emph{Lebesgue point} of $f$ if we have
\begin{equation}
	\label{LebesguePoint}
	\lim\limits_{r\rightarrow 0+}\frac{1}{\lambda(M_r(x,y)))}\int_{(s,t)\in M_r}|f(s,t)-f(x,y)|\differential(\lambda^2)=0,
\end{equation}
where $M_r(x,y) = [x-r,x+r]\times [y-r,y+r]$.
Recall that $(x,y)\in\RR^2$ is a \emph{point of density} of a set $A\subseteq \RR^2$ if
\begin{equation}
\label{density}
\lim\limits_{r\rightarrow 0+}\frac{\lambda(M_r(x,y)\setminus A)}{\lambda(M_r(x,y)))}=0.
\end{equation}
Recall also that a function $f\colon (0,1)^2\rightarrow\RR$ is said to be \emph{approximately continuous} at $(x,y)\in(0,1)^2$ if for every $\varepsilon>0$, the point $(x,y)$ is a point of density of the set $\{(s,t)\in(0,1)^2 \colon |f(s,t)-f(x,y)|<\varepsilon\}$. 

We will use the following classical result (see e.g. \cite[Theorem 7.7]{Rudin}).

\begin{theorem}
	\label{theorem:LP}
	Let $f\colon (0,1)^2\rightarrow\RR$ be an integrable function. Then almost every point of $(0,1)^2$ is a Lebesgue point of $f$. In particular, we have
\begin{equation*}
	\lim\limits_{r\rightarrow 0+}\frac{1}{\lambda(M_r(x,y)))}\int_{(s,t)\in M_r}f(s,t)\differential(\lambda^2)=f(x,y)
\end{equation*}
for almost every $(x,y)\in (0,1)^2$.\qed
\end{theorem}

An easy consequence of the previous theorem is also the following classical result.

\begin{theorem}
\label{theorem:appcont}
Let $f\colon (0,1)^2\rightarrow\RR$ be a measurable function. Then $f$ is approximately continuous at almost every point of $(0,1)^2$.\qed
\end{theorem}

\subsection{Alternative formulas for $\kappa(W)$}\label{ssec:AltKappa}
First, let us prove Fact~\ref{fact:kappaalt} which gives our first identity for $\kappa(W)$.
\begin{proof}[Proof of Fact~\ref{fact:kappaalt}]
If there exists a nonzero histogram $h$ such that $\int_{(u,v)\in\Omega^2} h(u)h(v)\log W(u,v)\differential(\nu^2)=0$ then clearly both suprema in~\eqref{eq:defKappaSimple} and \eqref{eq:defKappa} are $+\infty$.

Let $h$ be an arbitrary histogram in~\eqref{eq:defKappaSimple} and suppose that $\int_{(u,v)\in\Omega^2} h(u)h(v)\log\left(\nicefrac{1}{W(u,v)}\right)\differential(\nu^2)>0$. We take $$f=\frac{2\|h\|_1\cdot h}{\int_{(u,v)\in\Omega^2} h(u)h(v)\log\left(\nicefrac{1}{W(u,v)}\right)\;\differential(\nu^2)}\;.$$ This way, the function $f$ is a histogram on $\Omega$, and $\|f\|_1$ equals to the term in the supremum in~\eqref{eq:defKappaSimple}. So, to justify that the right-hand side of~\eqref{eq:defKappa} is at least as big as that of~\eqref{eq:defKappaSimple}, we only need to show that $\Gamma(f,W)\ge 0$. Indeed,
\begin{align*}
\Gamma(f,W)&=\tfrac{2\|h\|_1}{\int_{(u,v)\in\Omega^2} h(u)h(v)\log\left(\nicefrac{1}{W(u,v)}\right)\differential(\nu^2)}\cdot\int_{x\in\Omega}  h(x)\differential(\nu)\\
&~~~+\frac{1}2\cdot\left(\tfrac{2\|h\|_1}{\int_{(u,v)\in\Omega^2}h(u)h(v)\log\left(\nicefrac{1}{W(u,v)}\right)\differential(\nu^2)}\right)^2\cdot\int_{(x,y)\in\Omega^2} h(x)h(y)\log W(x,y)\differential(\nu^2)\\
&=0\;.
\end{align*}

On the other hand, let $f$ be an arbitrary histogram appearing in~\eqref{eq:defKappa} such that $$\int_{(u,v)\in\Omega^2} f(u)f(v)\log\left(\nicefrac{1}{W(u,v)}\right)\differential(\nu^2)>0\;.$$
For $c\in\RR$, let us denote by $cf$ the $c$-multiple of the function $f$.
Then the map $c \mapsto \Gamma(cf,W)$ is clearly a quadratic function with the limit $-\infty$ at $+\infty$.
And since $\Gamma(f,W)\geq0$, there is $c_0\ge 1$ such that $\Gamma(c_0f,W)=0$.
Now if we define $h=c_0f$ then the corresponding term in~\eqref{eq:defKappaSimple} equals to $\|h\|_1=c_0\|f\|_1\ge\|f\|_1$.
\end{proof}

In Fact~\ref{fact:kappaalt} we gave an alternative formula for $\kappa(W)$. In Proposition~\ref{proposition:differentkappa} we give two further expressions. These expressions require the graphon $W$ to be changed on a nullset. Note that we have the liberty of making such a change as the distribution of the model $\RG(n,W)$ remains unaltered.
\begin{proposition}\label{proposition:differentkappa}
Suppose that $W: (0,1)^2\rightarrow[0,1]$ is an arbitrary graphon. Then $W$ can be changed on a nullset in such a way that the following holds:

We have
	\begin{align}
	\label{eq:kappa1}
	\exp\left(-\frac2{\kappa(W)}\right)&=\lim_{r\rightarrow \infty} P_r\quad \mbox{where} \quad P_r=\sup_{F\subset (0,1), |F|=r} \left(\prod_{x,y\in F, x<y} W(x,y)\right)^{2/(|F|^2-|F|)}\;\mbox{and}\\
	\label{eq:kappa2}
	\frac{2}{\kappa(W)}&=\inf_{A}\frac{1}{\lambda(A)^2}\int_{(x,y)\in A^2}\log\left( \nicefrac1{W(x,y)}\right)\differential(\lambda^2)\;,
	\end{align}
	where the infimum ranges over all measurable sets $A\subset (0,1)$ of positive measure.

Moreover, \begin{equation}\label{eq:defKappaSets}
\kappa(W)=\sup_{a\in\mathbb R_+,A\subset\Omega}\left\{a\lambda(A) \::\: \Gamma(a\cdot\mathbf{1}_A,W)\ge 0\right\}\;.
\end{equation}
More precisely, for each $\epsilon>0$ and each set $A$ from~\eqref{eq:kappa2} satisfying $(1+\epsilon)\frac{2}{\kappa(W)}\ge  \frac{1}{\lambda(A)^2}\int_{(x,y)\in A^2}\log(\nicefrac{1}{W(x,y)})\differential(\lambda^2)$ we have that for the number $a=(1-\epsilon)\frac{\kappa(W)}{\lambda(A)}$ that $\Gamma(a\cdot\mathbf{1}_A,W)\ge 0$.
\end{proposition}
\begin{remark*}\Referee{measures}
We can rewrite~\eqref{eq:kappa2} as $\frac{2}{\kappa(W)}=\inf_{h}\frac{1}{\|h\|_1^2}\int_{(x,y)\in \Omega^2}\log \left(\nicefrac1{W(x,y)}\right)h(x)h(y)\differential(\lambda^2)$, where $h$ ranges over all indicator functions on~$(0,1)$. Our proof of Proposition~\ref{proposition:differentkappa} could be easily modified to show that in the infimum, we can range over all histograms~$h$ instead. Since the Radon--Nikodym theorem gives us a one-to-one correspondence between non-negative $L^1$-functions and absolutely continuous measures, we get that
$$
	\frac{2}{\kappa(W)}=\inf_{\pi}\int_{(x,y)\in (0,1)^2}\left(\log \nicefrac1{W(x,y)}\right)\differential(\pi^2)\;,
$$
	where the infimum ranges over all probability measures $\pi$ on $(0,1)$ that are absolutely continuous with respect to the Lebesgue measure. While we shall not need this identity we remark that it can be used to derive a heuristic --- slightly different from that given in Section~\ref{sec:formulaheuristic} --- for Theorem~\ref{thm:main}.
\end{remark*}
\begin{proof}[Proof of Proposition~\ref{proposition:differentkappa}]
	Let us replace the value of $W$ at every point $(x,y)\in(0,1)^2$ that is not a point of approximate continuity by $c$. This is a change of measure zero by Theorem~\ref{theorem:appcont}.
	
	Let us deal with the first part of the statement postponing~\eqref{eq:defKappaSets} to later.
			
	\begin{claim}\label{cl:step1}
		For each $r\in\{2,3,\ldots\}$ we have $\log P_r\ge -\frac{2}{\kappa(W)}$.
	\end{claim}
	\begin{proof}[Proof of Claim~\ref{cl:step1}]
	Let $h$ be an arbitrary function appearing in~\eqref{eq:defKappaSimple} (not constant zero). Fix $r\in\{2,3,\ldots\}$, and let $F\subset (0,1)$ be a random set consisting of $r$ independent points sampled from $(0,1)$ according to the density function $d=h/\|h\|_1$. Then by linearity of expectation we have
	\begin{align*}
	\Exp_{F}\left[\tfrac{2}{r(r-1)}\sum_{x,y\in F,x<y}\log W(x,y)\right]=\Exp_{x,y\sim d}\left[\log W(x,y)\right]=\frac{1}{\|h\|_1^2}\int_{(x,y)\in\Omega^2} h(x)h(y)\log W(x,y)\;\differential(\lambda^2)\;.
	\end{align*}
	This shows that there exists a deterministic $r$-element set $F$ for which $$\frac{2}{r(r-1)}\sum_{x,y\in F,x<y}\log W(x,y)\ge \frac{1}{\|h\|_1^2}\int_{(x,y)\in(0,1)^2} h(x)h(y)\log W(x,y)\differential(\lambda^2)\;.$$
	This concludes the proof of Claim~\ref{cl:step1}.
	\end{proof}

	Let us denote the right-hand side of~\eqref{eq:kappa2} as $-P$.

	\begin{claim}\label{cl:step2}
		For each $r\in\{2,3,\ldots\}$, we have that $P\ge \frac{r-1}{r}\log P_r+\frac{\log c}{r}$.
	\end{claim}
	\begin{proof}[Proof of Claim~\ref{cl:step2}]
	Suppose that $r\in\{2,3,\ldots\}$ is given, and let $F=\{x_1<x_2<\ldots<x_r\}$ be an arbitrary set of points in $(0,1)$ as in~\eqref{eq:kappa1}. Let $\epsilon\in(0,c/2)$ be arbitrary. Let $C=\log c-\log(c-\epsilon)$. Firstly, note that the concavity of the logarithm gives that 
	\begin{equation}\label{eq:logconc}
	\log(a-\epsilon)\ge\log a -C
	\end{equation} for each $a\in[c,\infty)$. Secondly, note that $C\searrow 0$ as $\epsilon\searrow 0$.
	
	Let us take $\delta>0$ such that for each $i\in[r]$, we have that the sets $S_i=[x_i-\delta,x_i+\delta]$ are pairwise disjoint, and such that the measure of each of the sets $D_{ij}=\{(x,y)\in S_i\times S_j: W(x_i,x_j)-\epsilon>W(x,y) \}$ is at most $\epsilon (2\delta)^2$. The latter property can be achieved since each point $(x_i,x_j)$ is either a Lebesgue point of $W$, or it is a point attaining the infimum of $W$. Let $A=\bigcup_{i=1}^r S_i$. Then,
	\begin{align*}
	\frac{1}{\lambda(A)^2}&\int\limits_{(x,y)\in A^2}\log W(x,y)\differential(\lambda^2)\\
	&=\frac{1}{(2\delta r)^2}\sum_i\sum_j\left(\int\limits_{(x,y)\in (S_i\times S_j)\setminus D_{ij}}\log W(x,y)\;\differential(\lambda^2)+\int\limits_{(x,y)\in D_{ij}}\log W(x,y)\;\differential(\lambda^2)\right)\\
	&\ge
	\frac{1}{(2\delta r)^2}\sum_i\sum_j\left(\int\limits_{(x,y)\in (S_i\times S_j)\setminus D_{ij}}\log (W(x_i,x_j)-\epsilon)
\;\differential(\lambda^2)+\int\limits_{(x,y)\in D_{ij}}\log c\;\differential(\lambda^2)\right)\;.
	\end{align*}
Now, let us use~\eqref{eq:logconc} for the first term and the fact that $\lambda(D_{ij})\le \epsilon(2\delta)^2$ for the second term.
Thus,
\begin{align*}
	\frac{1}{\lambda(A)^2}&\int_{(x,y)\in A^2}\log W(x,y)\differential(\lambda^2)\\&\ge
	\frac{1}{r^2}\sum_i\sum_j\log W(x_i,x_j)\;-C+\epsilon\log c\\
	&\ge 
	\frac{r-1}{r}\cdot\frac{2}{r(r-1)}\sum_{i,j:i<j}\log W(x_i,x_j)\;+\frac{1}{r^2}\sum_i\log W(x_i,x_i)-C+\epsilon\log c\\
	&\ge 
	\frac{r-1}{r}\cdot\frac{2}{r(r-1)}\sum_{i,j:i<j}\log W(x_i,x_j)\;-C+(\epsilon+\tfrac1r)\log c\;.
	\end{align*}
	Letting $\epsilon\searrow 0$ (which means that also $C\searrow0$), we get the claim.
	\end{proof}

By Claim~\ref{cl:step1}, we have $\liminf_{r\rightarrow\infty} \log P_r\ge-\tfrac{2}{\kappa(W)}$. By Claim~\ref{cl:step2}, we have $P\ge\limsup_{r\rightarrow\infty} \log P_r$. Further, it is obvious that $-\frac{2}{\kappa(W)}\ge P$: Indeed, the supremum in~\eqref{eq:defKappaSimple} ranges over all histograms, of which indicators of measurable sets are just a particular case. The combination of the three above inequalities proves the fact.
	
So, it remains to deal with~\eqref{eq:defKappaSets}. Positive multiples of indicator functions are histograms, so~\eqref{eq:defKappa} tells us that $\kappa(W)\ge\sup_{a\in\mathbb R_+,A\subset\Omega}\left\{a\lambda(A) \::\: \Gamma(a\cdot\mathbf{1}_A,W)\ge 0\right\}$. It remains to deal with the opposite inequality. We shall prove this in the ``more precisely'' form. Let $\epsilon>0$ be arbitrary and take $A$ such that $(1+\epsilon)\frac{2}{\kappa(W)}\ge \frac{1}{\lambda(A)^2}\int_{(x,y)\in A^2} \log(\nicefrac{1}{W(x,y)})\differential(\lambda^2)$. Set $a=(1-\epsilon)\frac{\kappa(W)}{\lambda(A)}$. We claim that the pair $(a,A)$ is admissible for the supremum in~\eqref{eq:defKappaSets}. Indeed,\Referee{13}
\begin{align*}
\Gamma(a\cdot\mathbf{1}_A,W)&\eqByRef{eq:A}a\lambda(A)-\tfrac12a^2\int_{(x,y)\in A^2}-\log W(x,y)\differential(\lambda^2)\\
&=(1-\epsilon)\kappa(W)\left(1-\tfrac12(1-\epsilon)\kappa(W)\frac{1}{\lambda(A)^2}\int_{(x,y)\in A^2}\log(\nicefrac{1}{W(x,y)})\differential(\lambda^2)\right)\\
&\ge
(1-\epsilon)\kappa(W)\left(1-\tfrac12(1-\epsilon)\kappa(W)(1+\epsilon)\frac{2}{\kappa(W)}\right)=\epsilon^2(1-\epsilon)\kappa(W)>0\;.
\end{align*}
Since $a\cdot \lambda(A)=(1-\epsilon)\kappa(W)$, and since $\epsilon>0$ was arbitrary, this finishes the proof.
\end{proof}

\subsection{Exhaustion principle}
We recall the principle of exhaustion (see e.g. \cite[Lemma 11.12]{FHHMZ} for a more general statement).

\begin{lemma}
\label{lemma:exhaustion}
Let $\mathcal C$ be a collection of measurable subsets of $(0,1)$ with positive Lebesgue measure.
Suppose that for every $A\subseteq (0,1)$ with positive Lebesgue measure, there is $C\in\mathcal C$ such that $C\subseteq A$.
Then there is an at most countable subcollection $\mathcal B$ of $\mathcal C$ of pairwise disjoint sets such that $\sum\limits_{B\in\mathcal B}\lambda(B)=1$. \qed
\end{lemma}

\section{Concentration and oscillation: Proofs of Proposition~\ref{prop:wild} and Theorem~\ref{thm:concentration}}\label{sec:examples}
\subsection{Proof of Proposition~\ref{prop:wild}}
 We shall need the following well-known crude bound on the minimum difference between uniformly random points.
\begin{fact}\label{fact:minimumlarge}
Suppose that the numbers $x_1,\ldots,x_n$ are uniformly sampled from the interval $(0,1)$. Then asymptotically almost surely, $\min_{i\neq j}|x_i-x_j|>n^{-3}$.
\end{fact}

Consider a sequence of positive numbers $1=a_1> a_2> a_3> \ldots>0$, with $\lim_{n\to\infty}a_n=0$, to be determined later. Define a graphon $W:(0,1)^2\rightarrow [0,1]$ as 
$$W(x,y)=
\left\{
\begin{array}{lr}
0 &  \mbox{if $a_{2i-1}\ge |x-y|>a_{2i}$,}\\
1 &  \mbox{if $a_{2i}\ge |x-y|>a_{2i+1}$.}
\end{array}
\right.
$$
Let us show how to achieve~\eqref{eq:Psismall}. Suppose that numbers $a_1,\ldots,a_{2i-1}$ were already set. Fix an arbitrary number $n$ large enough such that $n^{-3}<a_{2i-1}$ and $f(n)>1+1/a_{2i-1}$. Then, set $a_{2i}:=n^{-3}$. We claim that with high probability, there is no set of $f(n)$ vertices in $\RG(n,W)$ forming a clique. Indeed, consider the representation of the vertices of $\RG(n,W)$ in the interval $[0,1]$. By Fact~\ref{fact:minimumlarge} we can assume that the mutual distances between these points are more than $a_{2i}$. Consider an arbitrary set $S\subset (0,1)$ of these points of size $f(n)$. By the pigeonhole principle there are two points $x,y\in S$ with $|x-y|\le 1/(f(n)-1)<a_{2i-1}$. On the other hand, $|x-y|>a_{2i}$. We conclude that $W(x,y)=0$, and thus $S$ does not induce a clique. 

Next, let us show how to achieve~\eqref{eq:Psilarge}. Suppose that numbers $a_1,\ldots,a_{2i}$ were already set. Fix a large number $n$. In particular, suppose that $n^{-3}<a_{2i}$ and $f(n)>\frac 2{a_{2i}}$, and let $a_{2i+1}:=n^{-3}$. Now, consider the process of generating vertices in $\RG(n,W)$. By the Law of Large Numbers, out of $n$ vertices, with high probability, at least $\frac 12a_{2i}n$ vertices fall in the interval $(\frac12-\frac{a_{2i}}2,\frac12+\frac{a_{2i}}2)$. By Fact~\ref{fact:minimumlarge}, with high probability, the differences of pairs of these points are bigger than $a_{2i+1}$. In particular, the said set of vertices forms a clique of order at least $\frac 12a_{2i}n> \frac{n}{f(n)}$, as needed for~\eqref{eq:Psilarge}.\qed

\begin{remark}
It may seem that by replacing the values 0 and 1 by some constants $0<p_1<p_2<1$ in the construction in the proof of Proposition~\ref{prop:wild} we get an oscillation between $c_1\log n$ and $c_2\log n$. Theorem~\ref{thm:main} tells us however that this is not the case: the clique number normalized by $\log n$ will converge in probability.\qed
\end{remark}

\subsection{Proof of Theorem~\ref{thm:concentration}}
The proof of Theorem~\ref{thm:concentration} was suggested to us by Lutz Warnke.\medskip

First, we handle the case when $\Exp[\omega(\RG(n,W))]$ is bounded.
\begin{lemma}\label{lem:boundedcharacterization}
Let $W$ be a graphon, and $L=\lim_{n\to\infty} \Exp[\omega(\RG(n,W))]$. Then
$L=\sup \{k\in\NN:t(K_k,W)>0\}$. In addition, if $L$ is finite then $\lim_{n\to\infty} \Pr[\omega(\RG(n,W))=L]=1$.
\end{lemma}
\begin{proof}
The statement follows from the following claim. Suppose that $W$ is a
graphon. Then for each $\ell\in\NN$ we have
\begin{equation*}
\lim_{n\to\infty}\Pr[\omega(\RG(n,W))\ge \ell]\in \{0,1\}\;.
\end{equation*}
Indeed, suppose that for some $\ell$ and $n$ we have that $\Pr[\omega(\RG(n,W))\ge \ell]=\delta>0$. Then, for each $k$, we have that $\Pr[\omega(\RG(kn,W))\ge \ell]\ge 1-(1-\delta)^k$. Consequently, $\lim_{n\to\infty}\Pr[\omega(\RG(n,W))\ge \ell]=1$.\Referee{18}
\end{proof}

Suppose that $W$ is represented on a probability space $\Omega$.
By Lemma~\ref{lem:boundedcharacterization}, we can assume that $\Exp[\omega(\RG(n,W))]\rightarrow \infty$.

To prove the concentration, we shall use Talagrand's inequality. For this, we need to represent $\RG(n,W)$ on a suitable product space $J$. It turns out that the right product space corresponds to ``vertex-exposure'' technique known in the theory of random graphs. Let 
$J:=\prod_{i=1}^nJ_i$, where $J_i=\Omega\times [0,1]^{i-1}$. 
This indeed is a ``vertex-exposure model'' of $\RG(n,W)$. To see this, consider an arbitrary element ${\bf{x}}\in J$. We can write
$${\bf{x}}=(x_1,(\;); x_2,(\begin{matrix}
p_{1,2}
\end{matrix});
x_3,\left(\begin{matrix}
p_{1,3}\\
p_{2,3}
\end{matrix}\right);
\ldots;x_n,\left(\begin{matrix}
p_{1,n}\\
p_{2,n}\\
\cdots\\
p_{n-1,n}
\end{matrix}\right))\;,$$
where $x_i\in \Omega$, and $p_{i,j}\in[0,1]$. In the instance of $\RG(n,W)$ corresponding to ${\bf{x}}$, vertices $i$ and $j$ are connected if and only if $W(x_i,x_j)\ge p_{i,j}$. It is straightforward to check that this gives the right distribution on $\RG(n,W)$.

Consider the clique number, this time on the domain $J$. That is, we have a function $\Psi:J\to \RR$, where $\Psi({\bf{x}})$ is the clique number of the graph corresponding to ${\bf{x}}$. Then $\Psi$ is a (discrete) $1$-Lipschitz function. That is if ${\bf{x}},{\bf{y}}\in J$ are such that they differ in one coordinate, then $|\Psi({\bf{x}})-\Psi({\bf{y}})|\le 1$.\footnote{One could consider a stronger notion of $1$-Lipschitzness, namely, to require that changing ${\bf{x}}$ on an $E$-coordinate by $\epsilon$ would change the value of our function by at most $\epsilon$. This clearly is not true for $\Psi$. However, the weaker version is sufficient for our purposes.}
Further, $\Psi$ satisfies the so-called \emph{small certificates condition}. This means that whenever $\Psi({\bf{x}})\ge \ell$, there exists a set $C$ of at most $\ell$ many coordinates such that $\Psi({\bf{y}})\ge \ell$ for each ${\bf{y}}\in J$ which agrees with ${\bf{x}}$ on each coordinate from $C$. (In other words, the values of ${\bf{x}}$ on coordinates from $C$ alone certify that $\Psi({\bf{x}})\ge \ell$.) Indeed, it is enough just to reveal the values at the indices of one maximum clique. Talagrand's inequality (see~\cite[Remark 2 following Talagrand's Inequality~II, p.~81]{MolloyReed})\footnote{Actually, as was communicated to us by Lutz Warnke and Mike Molloy, there is a typo in~\cite{MolloyReed}. The effect of this typo, however, is only the value of the constant $\beta$ below. Since we do not make $\beta$ explicit, this typo is irrelevant.} then states that there exists an absolute constant $\beta>0$ such that for $t_n=\left(\Exp[\omega(\RG(n,W))]\right)^{\frac 34}$, we have (for every large enough $n$) that
$$\Pr\big[|\omega(\RG(n,W))-\Exp[\omega(\RG(n,W))]|>t_n\big]\le 2\exp\left(-\frac{\beta t_n^2}{\Exp[\omega(\RG(n,W))]}\right)=2\exp\left(-\beta \sqrt{\Exp[\omega(\RG(n,W))]}\right)\;.$$
The conclusion immediately follows by letting $n$ go to infinity.\qed

\subsection{Graphons with a bounded clique number}\label{ssec:boundedclique}
Lemma~\ref{lem:boundedcharacterization} provides some information about graphons with a bounded clique number. In this section, we prove Theorem~\ref{thm:boundedcliquenumber} which gives a much more explicit description.

\begin{proof}[Proof of Theorem~\ref{thm:boundedcliquenumber}]
Let $W\colon\Omega^2\rightarrow [0,1]$ be a graphon with a bounded clique number.
By Lemma~\ref{lem:boundedcharacterization} we know that $L=\sup \{k\in\NN:t(K_k,W)>0\}$ is a finite (natural) number.
First of all, we will show that there is a set $B\subseteq\Omega$ of positive measure such that $W\restriction_{B\times B}=0$ almost everywhere.
We may assume that $L\geq 2$ (the case $L=1$ is trivial).
For every ($L-1$)-tuple $\boldsymbol{x}=(x_1,\ldots,x_{L-1})\in \Omega^{L-1}$, let us denote $Q_{\boldsymbol{x}}=\{y\in \Omega \colon \prod_{i=1}^{L-1}W(x_i,y)>0\}$.
It follows from the equality $t(K_{L+1},W)=0$ that for every (up to a set of measure zero) $\boldsymbol{x}\in \Omega^{L-1}$ such that $\prod_{i<j}W(x_i,x_j)>0$, we have $W\restriction_{Q_{\boldsymbol{x}}\times Q_{\boldsymbol{x}}}=0$ almost everywhere.
But since $t(K_L,W)>0$, the set of all $\boldsymbol{x}\in \Omega^{L-1}$ such that $\prod_{i<j}W(x_i,x_j)>0$ and $\nu(Q_{\boldsymbol{x}})>0$, has positive measure.
So it suffices to set $B=Q_{\boldsymbol{x}}$ for a suitable $\boldsymbol{x}\in\Omega^{L-1}$.

Next, observe that for every $A\subset \Omega$ of positive measure, there is $B\subset A$ of positive measure such that $W\restriction_{B\times B}=0$ almost everywhere.
This follows by the previous considerations applied on the subgraphon $W^*=W\restriction_{A\times A}$ (for which we still have that $\sup \{k\in\NN:t(K_k,W^*)>0\} < + \infty$). 

Finally, let $W'\colon(0,1)^2\rightarrow [0,1]$ be a representation of the graphon $W$ on $(0,1)$. Then the statement follows by an application of Lemma~\ref{lemma:exhaustion}.
\end{proof}

\section{Biclique number}\label{sec:bislique}
In this section, we prove Theorem~\ref{thm:biclique}. First, we introduce some additional notation. When we refer to a bipartite graph $H=(V,W;E)$ as a \emph{bigraph}, we consider a distinguished order of the colour classes $V=\{v_1,\ldots,v_p\}$ and $W=\{w_1,\ldots,w_q\}$. In such a case we define the \emph{bipartite density} of $H$ in a bigraphon $U:\Omega_1\times \Omega_2\rightarrow [0,1]$ by
\begin{equation}
\label{eq:tB}
\tB(H,U)=\int_{(x_1,\ldots,x_p)\in \Omega_1^p}\int_{(y_1,\ldots,y_q)\in\Omega_2^q}
\prod_{ij:v_iw_j\in E}U(x_i,y_j)\;\differential(\nu_2^q)\differential(\nu_1^p)\;.
\end{equation}
Note that for a bigraph $H=(V,W;E)$ and its conjugate $H'=(W,V;E)$ the quantities $\tB(H,U)$ and $\tB(H',U)$ are not necessarily equal.

\medskip

As we will see, the upper bound in Theorem~\ref{thm:biclique} is trivial. For the lower bound, we need to make a small detour to Sidorenko's conjecture.
\subsection{Sidorenko's conjecture}
A famous conjecture of Simonovits and Sidorenko, ``Sidorenko's conjecture'',~\cite{Simonovits:Sidorenko,Sidorenko:Conjecture} asserts that among all graphs of a given (large) order~$n$ and fixed edge density~$d$, the density of a fixed bipartite graph is minimized for a typical sample from $\RG(n,d)$. The conjecture can be particularly neatly phrased in the language of graphons --- as observed already by Sidorenko a decade before the notion of graphons itself --- in which case it asserts that
\begin{equation}\label{eq:Sidorenko}
\tB(H,U)
\ge \left(\int_{x\in\Omega_1}\int_{y\in\Omega_2} U(x,y)\;\differential(\nu_2)\differential(\nu_1)\right)^{e(H)}\;,
\end{equation}
for each bigraphon $U:\Omega_1\times \Omega_2\rightarrow [0,1]$ and each bigraph $H$. Despite recent results~\cite{Hatami:Sidorenko,CoFoSu:Sidorenko,Lovasz:Sidorenko,LiSze:Sidorenko}, the conjecture is wide open.
We shall need the solution of Sidorenko's Conjecture for $H=K_{n,m}$ which was observed already by Sidorenko. We give a short self-contained and unoriginal proof here.
\begin{proposition}\label{prop:bipSidorenkoComplete}
Suppose that $U:\Omega_1\times \Omega_2\rightarrow [0,1]$ is an arbitrary bigraphon and $n,m\in\NN$ are arbitrary. Then for the complete bigraph $K_{n,m}$ we have
$\tB(K_{n,m},U)\ge \left(\int_{x\in\Omega_1}\int_{y\in\Omega_2} U(x,y)\;\differential(\nu_2)\differential(\nu_1)\right)^{nm}$.

\end{proposition}
\begin{proof}
We have
\begin{align*}
\int_{x_1,\ldots,x_n}\int_{y_1,\ldots,y_m}\;\prod_{i\in[n],j\in[m]}U(x_i,y_j)
&=\int_{x_1,\ldots,x_n}\Big(\int_{y}\prod_{i\in[n]}U(x_i,y)\Big)^m
\geq\Big(\int_{x_1,\ldots,x_n}\int_{y}\prod_{i\in[n]}U(x_i,y)\Big)^m\\
&=\Big(\int_{y}\int_{x_1,\ldots,x_n}\prod_{i\in[n]}U(x_i,y)\Big)^m
=\Big(\int_{y}\Big(\int_{x}U(x,y)\Big)^n\Big)^m\\
&\geq\Big(\int_{x}\int_{y}U(x,y)\Big)^{n m}\;,
\end{align*}
where both inequalities follow by applications of H\"older's inequality.
\end{proof}

\subsection{Bicliques in almost constant bigraphons}
As a preliminary step for our proof of Theorem~\ref{thm:biclique}, we study bicliques in random bipartite graphs sampled from almost constant bigraphons. This condition is formalized by the following definition.
\begin{definition}
A bigraphon $U:\Omega_1\times \Omega_2\rightarrow[0,1]$ is \emph{$(d,\epsilon)$-constant} if $\int_{(x,y)\in\Omega_1\times \Omega_2} U(x,y)\differential(\nu_1\times\nu_2)\ge d$ and $\esssup U\le d+\epsilon$.
\end{definition}
\begin{proposition}\label{prop:bicliquealmostconstant}
Let $0<d_1<d_2<1$ be given.
Then for every $\alpha\in(0,1)$ there exists $\epsilon\in(0,1)$ such that the following holds: Whenever we have $d\in(d_1,d_2)$ and a $(d,\epsilon)$-constant bigraphon $U:\Omega_1\times \Omega_2\rightarrow[0,1]$ then for $G\sim \BRG(k,U)$ we have a.a.s. that
$$\omega_2(G)\ge (1-\alpha)\cdot\frac{2}{\log 1/d}\cdot\log k\;.$$
\end{proposition}
\begin{proof}
Let $\alpha\in(0,1)$ be arbitrary. Suppose that $\epsilon>0$ is sufficiently small (we will make this precise later), $d\in(d_1,d_2)$ and $U:\Omega_1\times \Omega_2\rightarrow[0,1]$ is $(d,\epsilon)$-constant. Suppose further that $k$ is large. 

Let $X_k$ be the number of bicliques in $\BRG(k,U)$ whose two colour classes have size $\ell=(1-\alpha)\cdot\frac{2}{\log 1/d}\cdot\log k$. Multiplicities caused by automorphisms of $K_{\ell,\ell}$ are not counted. By Proposition~\ref{prop:bipSidorenkoComplete} we have
\begin{equation}\label{eq:firstm}
\Exp[X_k]={k\choose \ell}^2\cdot \tB(K_{\ell,\ell},U)
\ge {k\choose \ell}^2 d^{\ell^2}\ge \left(\frac{k}{\ell}\right)^{2\ell} d^{\ell^2}=\left(\frac{k^{2\alpha}}{\ell^2}\right)^\ell
\;.
\end{equation}
Next, we are going to show by a second moment argument that $X_k\approx \Exp[X_k]$ a.a.s. For $p,q=0,1,\ldots,\ell$, we define the bigraph $K_{[\ell,p,q]}$ as a result of gluing two copies of $K_{\ell,\ell}$ along $p$ vertices in the first colour class and $q$ vertices in the second colour class. Alternatively, $K_{[\ell,p,q]}$ can be obtained by erasing edges of two disjoint copies of the bigraph $K_{\ell-p,\ell-q}$ from $K_{2\ell-p,2\ell-q}$. We have 
\begin{equation}\label{eq:edgesKpq}
e(K_{[\ell,p,q]})=2\ell^2-pq\;.
\end{equation}
We have 
\begin{equation}\label{eq:secondmomentidentity}
\Exp[X_k^2]=\sum_{p=0}^\ell\sum_{q=0}^\ell \Exp[Y_{p,q}]\;,
\end{equation}
where $Y_{p,q}$ counts the number of bigraphs $K_{[\ell,p,q]}$ which preserve the order of the colour classes. We expand the second moment as
\begin{equation}\label{eq:defQ}
\Exp[Y_{p,q}]={k \choose \ell-p \:|\: \ell-p \:|\: p}\cdot{k \choose \ell-q \:|\: \ell-q \:|\: q}\cdot\tB(K_{[\ell,p,q]},U)\;.
\end{equation}
\begin{claim}\label{cl:pqlarge}
For every $c>0$ there exists $\epsilon_1>0$ such that the following holds: Whenever $d\in(d_1,d_2)$ and $U:\Omega_1\times \Omega_2\rightarrow[0,1]$ is $(d,\epsilon_1)$-constant then for each $p,q\in[\ell]_0$, $p+q\ge c\log k$, we have
$$\Exp[X_k]^2\ge \log^3 k\cdot \Exp[Y_{p,q}]\;.$$
\end{claim}
\begin{proof}[Proof of Claim~\ref{cl:pqlarge}]
Let $c>0$ be arbitrary.
Suppose that $\epsilon_1>0$ is sufficiently small, $d\in(d_1,d_2)$ and $U:\Omega_1\times \Omega_2\rightarrow[0,1]$ is $(d,\epsilon_1)$-constant. Let $p,q\in[\ell]_0$ be such that $p+q\ge c\log k$ (and of course $p+q\le 2\ell$). Upper-bounding the terms in~\eqref{eq:defQ} (while bearing in mind that $\ell=\Theta(\log k)$), we get
\begin{equation}
\label{eq:GMAM}
\begin{aligned}
\Exp[Y_{p,q}]&\le
 k^{4\ell-p-q}(d+\epsilon_1)^{e(K_{[\ell,p,q]})}
\eqByRef{eq:edgesKpq}
 k^{4\ell-p-q}(d+\epsilon_1)^{2\ell^2-pq}\\
\JUSTIFY{AM-GM Ineq., $d+\epsilon_1<1$} &\le
 \left(k^2(d+\epsilon_1)^{\ell}\right)^{2\ell}
 \left(k(d+\epsilon_1)^{\frac{p+q}4}\right)^{-(p+q)}\\
\JUSTIFY{$p+q\le 2\ell$ and $d+\epsilon_1<1$} &\le  \left(k^2(d+\epsilon_1)^{\ell}\right)^{2\ell}
 \left(k(d+\epsilon_1)^{\frac{\ell}2}\right)^{-(p+q)}\\
 \JUSTIFY{$k(d+\epsilon_1)^{\frac{\ell}2}\ge k^{\alpha}$} &\le
 \exp\left(4\ell\log k\left(1+(1-\alpha)\tfrac{\log(d+\epsilon_1)}{\log 1/d}\right)\right)k^{-c\alpha\log k}\\
 &=
 k^{4\alpha \ell}\exp\left(4\ell\log k\left(1-\alpha\right)\left(1-\tfrac{\log (d+\epsilon_1)}{\log d}\right)-c\alpha\log^2 k\right)\\
\JUSTIFY{$\epsilon_1$ sufficiently small}&\le 
k^{4\alpha \ell}\exp\left(-\tfrac 12c\alpha\log^2 k\right)
\end{aligned}
\end{equation}
It is now enough to compare this with~\eqref{eq:firstm}.
\end{proof}

\begin{claim}\label{cl:pqsmall}
There exist numbers $C,\epsilon_2>0$
such that the following holds: Whenever $d\in(d_1,d_2)$ and $U:\Omega_1\times \Omega_2\rightarrow[0,1]$ is $(d,\epsilon_2)$-constant then for each $p,q\in[\ell]_0$, $1\le p+q<C\log k$, we have
$$\Exp[Y_{0,0}]\ge k^{\frac12}\Exp[Y_{p,q}]\;.$$
\end{claim}
\begin{proof}[Proof of Claim~\ref{cl:pqsmall}]
Suppose that $\epsilon_2>0$ is sufficiently small, $d\in(d_1,d_2)$ and $U:\Omega_1\times \Omega_2\rightarrow[0,1]$ is $(d,\epsilon_2)$-constant.
Let us compare the combinatorial coefficients corresponding to $\Exp[Y_{0,0}]$ and $\Exp[Y_{p,q}]$ in~\eqref{eq:defQ}. We have 
\begin{equation}\label{eq:bclet}
\frac{{k\choose \ell\:|\:\ell}^2}{{k\choose \ell-p\:|\:\ell-p\:|\: p}\cdot{k\choose \ell-q\:|\:\ell-q\:|\: q}}=k^{(1+o(1))(p+q)}\;,
\end{equation}
where $o(1)\rightarrow 0$ as $k\rightarrow \infty$ uniformly for any choice of $p$ and $q$.

\medskip

It remains to compare the terms $\tB(K_{[\ell,0,0]},U)$ and $\tB(K_{[\ell,p,q]},U)$.

First, we claim that for each $i,j,h\in \NN$ we have 
\begin{equation}\label{eq:HarrisA}
\tB(K_{i,h},U)\tB(K_{j,h},U)\le \tB(K_{i+j,h},U)\;.
\end{equation}
Indeed, by H\"older's inequality, we have
\begin{align*}
\tB(K_{i,h},U)&=\int_T\left(\int_x \deg(x,T)\right)^i\leq \left(\int_T\left(\int_x \deg(x,T)\right)^{i+j}\right)^{\frac{i}{i+j}}\;,\\
\tB(K_{j,h},U)&=\int_T\left(\int_x \deg(x,T)\right)^j\leq \left(\int_T\left(\int_x \deg(x,T)\right)^{i+j}\right)^{\frac{j}{i+j}}\;,\\
\end{align*}
where the integrations are over $T=(t_1,\ldots,t_h)\in (\Omega_2)^h$, and $x\in \Omega_1$, and $\deg(x,T)=\prod_{r=1}^h U(x,t_r)$. Thus
\begin{align*}
\tB(K_{i,h},U)\tB(K_{j,h},U)\le \int_T\left(\int_x \deg(x,T)\right)^{i+j}=\tB(K_{i+j,h},U)\;,
\end{align*}
as we wanted.

A double application of~\eqref{eq:HarrisA} followed by an application of Proposition~\ref{prop:bipSidorenkoComplete} gives
\begin{align}
\begin{split}\label{eq:max1}
\tB(K_{\ell,\ell},U)&\ge \tB(K_{\ell-p,\ell},U) \tB(K_{p,\ell},U)
\ge \tB(K_{\ell-p,\ell-q},U) \tB(K_{\ell-p,q},U) \tB(K_{p,\ell},U)\\
&\ge \tB(K_{\ell-p,\ell-q},U) d^{q\ell-pq} d^{p\ell}\;.
\end{split}
\end{align}
In the defining formula~\eqref{eq:tB} for $\tB(K_{[\ell,p,q]},U)$ we show that $d+\epsilon_2$ are upper bounds on some factors in $\prod U(x_i,y_i)$, as in Figure~\ref{fig:deletededges}. Observe that after removal of the $p\ell+q\ell-2pq$ edges indicated in Figure~\ref{fig:deletededges}, the graph $K_{[\ell,p,q]}$ decomposes into a disjoint union of $K_{\ell,\ell}$ and $K_{\ell-p,\ell-q}$.
\begin{figure}[t]
    \centering
    \includegraphics{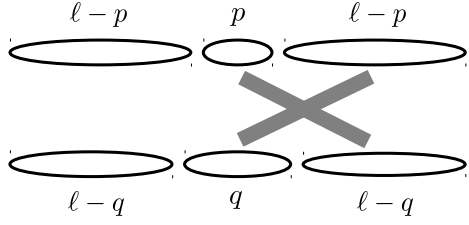}
     \caption{The $p\ell+q\ell-2pq$ edges of $K_{[\ell,p,q]}$ corresponding to factors bounded from above by $d+\epsilon_2$ in~\eqref{eq:tB}.}
\label{fig:deletededges}
\end{figure}
Note that for a disjoint union $H_1\oplus H_2$ of two bigraphs $H_1$ and $H_2$ we have $\tB(H_1\oplus H_2,U)=\tB(H_1,U)\tB(H_2,U)$.
Thus,
\begin{equation}\label{eq:max2}
\tB(K_{[\ell,p,q]},U)\le (d+\epsilon_2)^{p\ell+q\ell-2pq}\tB(K_{\ell,\ell},U) \tB(K_{\ell-p,\ell-q},U).
\end{equation}
Therefore, 
\begin{align}
\begin{split}\label{eq:poN}
\tB(K_{[\ell,0,0]},U)&=\tB(K_{\ell,\ell},U)\tB(K_{\ell,\ell},U)\geByRef{eq:max1} \tB(K_{\ell,\ell},U)\tB(K_{\ell-p,\ell-q},U)d^{p\ell+q\ell-pq}\\
&\geByRef{eq:max2}
\tB(K_{[\ell,p,q]},U)\left(\frac{d}{d+\epsilon_2}\right)^{p\ell+q\ell}
\left(\frac{(d+\epsilon_2)^2}d\right)^{pq}
\\
\JUSTIFY{AM-GM Ineq., $\epsilon_2$ such that $\frac{(d+\epsilon_2)^2}{d}\le 1$}&\ge 
\tB(K_{[\ell,p,q]},U)
\left(1-\frac{\epsilon_2}{d+\epsilon_2}\right)^{p\ell+q\ell}
\left(\frac{(d+\epsilon_2)^2}d\right)^{(p+q)\cdot\frac{p+q}4}\\
&\ge
\tB(K_{[\ell,p,q]},U)
\left(1-\frac{\epsilon_2}{d+\epsilon_2}\right)^{p\ell+q\ell}
d^{(p+q)\cdot \frac{C\log k}4}\;.
\end{split}
\end{align}
Substituting~\eqref{eq:bclet} and~\eqref{eq:poN} into~
\eqref{eq:defQ} we get
\begin{align*}
\frac{\Exp[Y_{0,0}]}{\Exp[Y_{p,q}]}&\ge \left(k^{1+o(1)}\left(1-\frac{\epsilon_2}{d+\epsilon_2}\right)^{\ell}d^\frac{C \log k}4\right)^{p+q}\\
\JUSTIFY{for $\zeta\in(0,\tfrac12)$ we have $1-\zeta\ge\exp(-2\zeta)$}&
\ge
\left(k^{1+o(1)}\exp\left(-(1-\alpha)\frac{2\epsilon_2}{d+\epsilon_2}
\cdot \frac{2}{\log 1/d}\cdot\log k\right)d^\frac{C \log k}4\right)^{p+q}\\
&=\left(k^{1+o(1)}\cdot k^{(1-\alpha)\frac{2\epsilon_2}{d+\epsilon_2}
\cdot \frac{2}{\log d}}\cdot k^\frac{C \log d}4\right)^{p+q}\\
&\ge k^{\frac12(p+q)}\ge k^{\frac12}\;,
\end{align*}
for $C=1/\log(1/d_1)$, and $\eps_2<\tfrac{1}{20} d_1\log(1/d_2)$.
\end{proof}
Let $C>0$ and $\eps_2>0$ be given by Claim ~\ref{cl:pqsmall}. Let $\eps_1>0$ be given by Claim~\ref{cl:pqlarge} for $c=C$. Now if $\eps<\min\left(\eps_1,\eps_2\right)$, we are ready provide upper bounds on the summands in~\eqref{eq:secondmomentidentity}. Note that we have $\Theta(\log^2  k)$ many of these summands.
We get
\begin{align*}
\Exp[X_k^2]&=\Exp[Y_{0,0}]+\sum_{\substack{p,q\in[l]_0\\ 1\le p+q<C\log k}}\Exp[Y_{p,q}]+\sum_{\substack{p,q\in[l]_0\\p+q\ge C\log k}}\Exp[Y_{p,q}]\\
&\le\Exp[Y_{0,0}]+\Theta(\log^2k)\Exp[Y_{0,0}]k^{-\tfrac 12}+\Theta(\log^2k)\Exp[X_k]^2\log^{-3}k\\
\JUSTIFY{$\Exp[Y_{0,0}]=(1+o(1))\Exp[X_k]^2$}&=(1+o(1))\Exp[X_k]^2.
\end{align*}
Therefore we have $\Exp[X_k^2]=(1+o(1))\Exp[X_k]^2$, and it follows by Chebyshev's inequality that $\Pr[X_k>0]=1-o(1)$.
\end{proof}
\subsection{Proof of Theorem~\ref{thm:biclique}}
The upper bound is easy, since it claims that $\omega_2(\BRG(n,U))$ is typically not bigger than the biclique number in the balanced bipartite Erd\H{o}s--R\'enyi random graph $\BRG(n,p)$ (which clearly stochastically dominates $\BRG(n,U)$). For completeness, we include the calculations. We write $Y_k(\BRG(n,U))$ for the number of complete balanced bipartite graphs on $k+k$ vertices inside $\BRG(n,U)$. For $k=(1+\epsilon)\cdot \frac{2}{\log 1/p}\cdot\log n$, we have,
\begin{align}\label{eq:bicliqueexp}
\Exp[Y_k]\le{n\choose k}\cdot {n\choose k}\cdot p^{k^2}
\le n^{2k}p^{k^2}=(n^2p^k)^k\;.
\end{align}
The statement now follows from Markov's Inequality, provided that we can show that $n^2p^k\rightarrow 0$. Indeed,
$$n^2p^k=n^2 p^\frac{2\log n}{\log 1/p}p^{\frac{2\epsilon\log n}{\log 1/p}}=p^{\frac{2\epsilon\log n}{\log 1/p}}\rightarrow 0\;.$$

\bigskip
We now turn to the lower bound. Let $\alpha\in(0,1)$ be arbitrary and let $\epsilon_0>0$ be given by Proposition~\ref{prop:bicliquealmostconstant} for $d_1=\tfrac p2$ and $d_2=p$. Let $\eps<\min\left(\epsilon_0,\tfrac p2\right)$ be arbitrary.
We denote by $\nu_i$ the measure given on $\Omega_i$, $i=1,2$.
The definition of the essential supremum, together with Theorem \ref{theorem:LP}, gives that there exist two measurable sets $A\subset \Omega_1$ and $B\subset \Omega_2$ such that $\nu_1(A),\nu_2(B)>0$ and $$\int_{(x,y)\in\Omega_1\times\Omega_2}U(x,y)\differential(\nu_1\times\nu_2)\ge(p-\epsilon)\nu_1(A)\nu_2(B)\;.$$ We put $\delta=\min\left(\nu_1(A),\nu_2(B)\right)$. By rescaling the measures $\nu_1,\nu_2$, we get probability measures $\nu_1^*$ on $A$ and $\nu_2^*$ on $B$. Then we can view~$U\restriction_{A\times B}$ as a bigraphon, which we denote by $U^*$. Note that $U^*$ is $(p-\eps,\epsilon)$-constant (and thus also $(p-\eps,\epsilon_0)$-constant as $\epsilon<\epsilon_0$).

Consider now the sampling process to generate $B\sim\BRG(n,U)$ as described above. A standard concentration argument gives that with high probability, at least $\frac{\delta n}2$ points $x_i$ sampled in $\Omega_1$ lie in the set $A$, and at least $\frac{\delta n}2$ points $y_j$ sampled in $\Omega_2$ lie in the set $B$. In other words, with high probability we can find a copy of $\BRG(\frac{\delta n}2,U^*)$ inside $\BRG(n,U)$. Looking at the biclique number, we get that for $\ell=(1-\alpha)\cdot \frac{2}{\log 1/(p-\eps)}\cdot \log(\delta n/2)$, we have
$$\Pr[\omega_2(\BRG(n,U))\ge \ell]\ge \Pr\left[\omega_2(\BRG(\tfrac{\delta n}2,U^*))\ge \ell\right]-o(1)\ge 1-o(1)\;,$$
where the last inequality follows from Proposition~\ref{prop:bicliquealmostconstant}. Since $\log(\delta n/2)=(1+o(1))\log n$, and since $\alpha\in(0,1)$ and $\eps<\min\left(\epsilon_0,\tfrac p2\right)$ were arbitrary, the claim follows.
\qed

\section{Formula for graphs with logarithmic clique number}\label{sec:formulaheuristic}
In this  section we try to informally justify Theorem~\ref{thm:main}. While we believe that the derivation here captures the essence of the problem, the actual proof, presented in Section~\ref{sec:mainproof}, is quite different. At the end of this section we comment on what fails in turning the heuristics into a rigorous argument. 

\subsection{First moment for a 2-step graphon}
\label{subsection:FirstMoment}
Let us try to gain some intuition on Theorem~\ref{thm:main} by looking at one of the simplest non-constant graphons. Let $W:\Omega^2\rightarrow [0,1]$ be represented on the unit interval $\Omega$ (c.f. Fact~\ref{fact:representinterval}) and defined by
\begin{equation}\label{eq:our2stepgraphon}
W(x,y)=\begin{cases} 
p_{11} &\mbox{if $x,y\in \Omega_1$}, \\
p_{22} & \mbox{if $x,y\in \Omega_2$, and}\\
p_{12} & \mbox{otherwise}.
\end{cases}
\end{equation}
Here, $p_{11},p_{22},p_{12}\in(0,1)$ are arbitrary, and $\Omega$ is partitioned arbitrarily into two measurable sets $\Omega_1$ and $\Omega_2$ of positive measures $\beta_1$ and $\beta_2$. 

Our aim is to determine for which numbers $c\in \RR^+$ there typically exists a clique of order $c\log n$ in $G\sim\RG(n,W)$, and for which $c$'s there is typically none.
So let us fix $c\in \RR^+$ and let
$X_n$ count the number of cliques of order $c\log n$ in $G\sim \RG(n,W)$. By Markov's Inequality, $\omega(G)$ will be typically smaller than $c\log n$ in the regime when $\Exp[X_n]\rightarrow 0$. On the other hand, it is plausible that a second moment argument will give that typically $\omega(G)\ge c\log n$ when $\Exp[X_n]\rightarrow +\infty$. With this belief --- which is supported by the success of a second-moment argument in the proof of Theorem~\ref{thm:biclique} --- let us estimate $\Exp[X_n]$. Actually, we rather look at a refined quantity $Y^{\alpha_1,\alpha_2}_n(G)$ which is defined as the number of cliques in $G$ that consist of $\alpha_1\log n$ vertices whose representation on $\Omega$ lies in $\Omega_1$, and $\alpha_2\log n$ vertices that are represented in $\Omega_2$. We have
$$\Exp[X_n]=\sum_{m=0}^{c\log n}\Exp\left[Y^{m/\log n,c-m/\log n}_n\right]\;.$$
We expect the quantities $Y^{\alpha_1,\alpha_2}_n$ to be either super-polynomially small or super-polynomially large. Since the sum above has only $\Theta(\log n)$-many summands, we expect that
\begin{equation}\label{eq:ifandonly}
\Exp[X_n]\rightarrow +\infty \quad \text{if and only if}\quad \exists \alpha_1,\alpha_2\ge 0\text{ such that }\alpha_1+\alpha_2=c\text{ and }\Exp[Y^{\alpha_1,\alpha_2}_n]\rightarrow+\infty\;.
\end{equation}

For a clique that contributes to $Y^{\alpha_1,\alpha_2}_n(G)$ to be present, $\alpha_1\log n \choose 2$ edges in the $(\Omega_1\times \Omega_1)$-part of $W$, $\alpha_2\log n \choose 2$ edges in the $(\Omega_2\times \Omega_2)$-part, and $\alpha_1\alpha_2\log^2 n$ edges in the $(\Omega_1\times \Omega_2)$-part must be present in the specific locations of perspective cliques or complete bipartite graphs. By the Law of Large Numbers, approximately $\beta_1 n$ points in the sampling process for $\RG(n,W)$ end up in $\Omega_1$ and approximately $\beta_2 n$ points end up in $\Omega_2$. We get
\begin{align}
\begin{split}\label{eq:FirstMomentY}
\Exp[Y^{\alpha_1,\alpha_2}_n]&\approx {\beta_1 n\choose \alpha_1 \log n}{\beta_2 n\choose \alpha_2 \log n}p_{11}^{\alpha_1\log n \choose 2}p_{22}^{\alpha_2\log n \choose 2}p_{12}^{\alpha_1\alpha_2\log^2 n}\\
&= \exp\left((1+o(1))\log^2 n\left(\alpha_1+\alpha_2+\frac{\alpha_1^2}2\log p_{11}+\frac{\alpha_2^2}2\log p_{22}+\alpha_1\alpha_2\log p_{12}\right)\right)\;.
\end{split}
\end{align}
Thus, whether $\Exp[Y^{\alpha_1,\alpha_2}_n]\rightarrow 0$ or $\Exp[Y^{\alpha_1,\alpha_2}_n]\rightarrow +\infty$ depends on whether
\begin{equation}\label{eq:simpleEntEne}
\underbrace{
\vphantom{\frac12}
\alpha_1+\alpha_2}_{(*)}+\underbrace{\frac12\left({\alpha_1^2}\log p_{11}+{\alpha_2^2}\log p_{22}+\alpha_1\alpha_2\log p_{12}+\alpha_2\alpha_1\log p_{12}\right)}_{(**)}
\end{equation} 
is negative or positive, respectively. It is straightforward to generalize this formula to graphons with more steps.
Observe also that the values of $\beta_1$ and $\beta_2$ get lost in the transition between the first and the second line of~\eqref{eq:FirstMomentY}, and are immaterial in~\eqref{eq:simpleEntEne} consequently (provided that they are positive). In particular, the step sizes $\beta_i$ could have been ``infinitesimally small''. Thus, we can see a direct correspondence between~$(*)$ and~$(**)$ in~\eqref{eq:A} and~\eqref{eq:simpleEntEne}, where the integration corresponds to passing to infinitesimal steps. In view of this, we will denote the term in (\ref{eq:simpleEntEne}) by $\Gamma(\alpha,W)$, where $\alpha=(\alpha_1,\alpha_2)$. The optimization over $\alpha_1$ and $\alpha_2$ in~\eqref{eq:ifandonly} corresponds to taking the supremum in~\eqref{eq:defKappa}.

\medskip
This is why we call the functions $f$ in~\eqref{eq:defKappa} (or vectors, in case of step-graphons) histograms: they specify the densities of the vertices of the anticipated cliques over the space $\Omega$. Also, motivated by~\eqref{eq:defKappa} and its interpretation above, we say that a histogram $f$ is \emph{admissible} for a graphon $W$ if $\Gamma(f,W)\ge 0$.

Last, let us note that a physicist  might refer to~$(*)$ as the ``entropy contribution'', as it comes from the choice of the vertices of a clique, while~$(**)$ could be referred to as the ``energy'' needed to include all required edges of that clique.

\subsection{Introducing the second moment to the example}\label{ssec:IntroSecondMoment}
So far, our prediction was based on a first moment argument. Combined with Markov's Inequality this gives readily an upper bound on the typical clique number of $\RG(n,W)$. We now want to complement the upper bound with a lower bound based on a second moment argument. Let us first recall the situation in the setting of the Erd\H{o}s--R\'enyi random graphs $\RG(n,p)$. There, a straightforward calculation for the random variable $X_n$ counting cliques of order $c\log n$ (where $c>0$ is fixed) gives that $\Exp[X_n]^2=(1+o(1))\Exp[X_n^2]$ if and only if $\Exp[X_n]\rightarrow +\infty$. Thus, the first and the second moment start working together at the same time.

The situation is more complicated in the model $\RG(n,W)$. We will illustrate this on the graphon $W$ from~\eqref{eq:our2stepgraphon}. Suppose that $\alpha_1,\alpha_2\ge0$ are such that~\eqref{eq:simpleEntEne} is positive, and we ask in hope whether 
\begin{equation}\label{eq:gettingthesecondmoment}
\Exp[Y^{\alpha_1,\alpha_2}_n]^2=(1+o(1))\Exp[(Y^{\alpha_1,\alpha_2}_n)^2]
\end{equation}
In~\eqref{eq:FirstMomentY}, we provided asymptotics for $\Exp[Y^{\alpha_1,\alpha_2}_n]$.
Thus, to understand whether we have~\eqref{eq:gettingthesecondmoment}, we need to calculate $\Exp[(Y^{\alpha_1,\alpha_2}_n)^2]$. We have 
\begin{equation}
\label{eq:secondmomentBASIC}
\Exp[(Y^{\alpha_1,\alpha_2}_n)^2]=\Exp\left[\left(\sum_K \mathbf{1}_{\text{$K$ induces a clique}}\right)^2\right]=\sum_{K,L}\Pr[\text{$K$ and $L$ induce cliques}]\;,
\end{equation}
where $K$ and $L$ range over all sets of vertices with $\alpha_1 \log n$ vertices represented in $\Omega_1$ and $\alpha_2 \log n$ vertices represented in $\Omega_2$. Let $K_1$ be the vertices of $K$ represented in $\Omega_1$, and let $K_2$, $L_1$, and $L_2$ be defined analogously.
It is clear that $|K_1\setminus L_1|=|L_1\setminus K_1|$ and $|K_2\setminus L_2|=|L_2\setminus K_2|$.
So for fixed sets $K$ and $L$, we have
$$\Pr[\text{$K$ and $L$ induce cliques}]=p_{11}^{{|K_1\cup L_1|\choose 2}-|K_1\setminus L_1|^2}\cdot
p_{22}^{{|K_2\cup L_2|\choose 2}-|K_2\setminus L_2|^2}\cdot
p_{12}^{|K_1\cup L_1|\cdot|K_2\cup L_2|-2|K_1\setminus L_1||K_2\setminus L_2|}
\;.$$
Thus, grouping~\eqref{eq:secondmomentBASIC} depending on the values of $m_1=|K_1\cap L_1|$ and $m_2=|K_2\cap L_2|$ we get 
\begin{equation}
\begin{aligned}
\label{eq:E(Y^2)}
\Exp[(Y^{\alpha_1,\alpha_2}_n)^2]\approx\sum_{m_1=0}^{\alpha_1\log n}\;\sum_{m_2=0}^{\alpha_2\log n}
&{\beta_1 n\choose m_1\ |\ \alpha_1\log n-m_1\ |\ \alpha_1\log n-m_1}\\
\times&{\beta_2 n\choose m_2\ |\ \alpha_2\log n-m_2\ |\ \alpha_2\log n-m_2}\\
\times&
p_{11}^{{2\alpha_1\log n-m_1\choose 2}-(\alpha_1\log n-m_1)^2}\cdot
p_{22}^{{2\alpha_2\log n-m_2\choose 2}-(\alpha_2\log n-m_2)^2}\\
\times&
p_{12}^{(2\alpha_1\log n-m_1)\cdot(2\alpha_2\log n-m_2)-2(\alpha_1\log n-m_1)(\alpha_2\log n-m_2)}
\;,
\end{aligned}
\end{equation}
where the approximate equality represents the fact that we assumed on the right-hand side that exactly $\beta_i n$ vertices are represented in $\Omega_i$. Let us write $\gamma_i=m_i/\log n$, and let us write $z^{\gamma_1,\gamma_2}_n$ for the individual summands on the right-hand side of (\ref{eq:E(Y^2)}). Routine manipulations give that 
\begin{align*}
\frac{\log z^{\gamma_1,\gamma_2}_n}
{\log^2 n}&\approx
\gamma_1+\gamma_2+2(\alpha_1-\gamma_1)+2(\alpha_2-\gamma_2)\\
&+
(\alpha_1^2-\tfrac12\gamma_1^2)\log p_{11}+(\alpha_2^2-\tfrac12\gamma_2^2)\log p_{22}+(2\alpha_1\alpha_2-\gamma_1\gamma_2)\log p_{12}\;.
\end{align*}
Thus, if we want the second moment~\eqref{eq:gettingthesecondmoment} to work then comparing the calculations above with~\eqref{eq:FirstMomentY}, we must have for each $\gamma_1\in[0,\alpha_1]$ and
$\gamma_2\in[0,\alpha_2]$ that
\begin{align*}
2&\left(\alpha_1+\alpha_2+\frac{\alpha_1^2}2\log p_{11}+\frac{\alpha_2^2}2\log p_{22}+\alpha_1\alpha_2\log p_{12}\right)\\
&\gtrapprox 
\gamma_1+\gamma_2+2(\alpha_1-\gamma_1)+2(\alpha_2-\gamma_2)+
(\alpha_1^2-\tfrac12\gamma_1^2)\log p_{11}+(\alpha_2^2-\tfrac12\gamma_2^2)\log p_{22}+(2\alpha_1\alpha_2-\gamma_1\gamma_2)\log p_{12}\;,
\end{align*}
which rewrites as $\Gamma(\gamma,W)\gtrapprox 0$. To summarize, to justify~\eqref{eq:defKappa} (at least for step-functions), it would suffice to have the following.
\begin{dreamlemma*}
If $W$ is a step-graphon with $k$ steps and $\alpha\in\mathbb [0,+\infty)^k$ is a vector with $\Gamma(\alpha,W)> 0$, then for all $\gamma\in \BOX(\alpha)$ we have $\Gamma(\gamma,W)\ge 0$.
\end{dreamlemma*}
This, however, does not hold in general. Indeed, take for example
$$p_{11}=e^{-3}, p_{12}=p_{22}=e^{-\frac{1}{4}}, \quad \alpha_1=1, \alpha_2=1, \quad \gamma_1=1, \gamma_2=0.$$
We have $\Gamma(\alpha,W)=\frac{1}{8}>0$ but $\Gamma(\gamma,W)=-\frac{1}{2}<0$. 
It is worth explaining what is happening in the example in words. The parameters $p_{11}$ and $\alpha_1$ are set so that asymptotically almost surely, $\RG(\frac{n}2,p_{11})$ contains no cliques of order $\alpha_1\log n$. However, in the rare cases when $\RG(\frac{n}2,p_{11})$ (viewed as a subgraph of $\RG(n,W)$) does contain such a clique, there are typically many ways of extending it on the $\Omega_2$-part by $\alpha_2\log n$ additional vertices, thus inflating $\Exp[Y_n^{\alpha_1,\alpha_2}]$ substantially.

The above suggests a correction for~\eqref{eq:defKappa} in that we should range only over those histograms $f$ for which $\Gamma(g,W)\ge 0$ for all histograms $g\in\BOX(f)$. Note also that the \emph{necessity} of testing the admissibility condition over all sub-histograms of $f$ has a clear combinatorial interpretation: If cliques with a given histogram typically appear, then for each given sub-histogram cliques with that sub-histogram must appear (just because a subset of a clique again induces a clique).

\medskip
Now, after all the arguing why \eqref{eq:defKappa} should seem wrong, let us explain why it is actually all right. We show in Lemma~\ref{lem:maxl1} that for any histogram $f$ attaining the supremum in~\eqref{eq:defKappa} we automatically have that all sub-histograms are admissible (recall that a histogram $h$ is admissible for a graphon $W$ if $\Gamma(h,W)\ge 0$). If the supremum is not attained then for each histogram $f$ almost attaining the supremum, we have $\Gamma(g,W)\gtrapprox 0$ for all sub-histograms $g$, which is sufficient for the argument.

\subsection{Failure of turning the above heuristics into a rigorous argument}\label{ssec:failure}
There are two types of errors that we introduced in the above argument.  Firstly, the ``little imprecisions'' when we replaced a sum by its maximal term (such as in~\eqref{eq:ifandonly}) or when we used the $\lessapprox$-symbol. Each such step introduces an error of $o(1)$ to the quantities $\frac{\log X_n}{\log^2 n}$ and $\frac{\log Y_n^{\alpha_1,\alpha_2}}{\log^2 n}$. That means, that actually we can only conclude that
$$\Exp[Y^{\alpha_1,\alpha_2}_n]^2=
\exp\left(o(\log^2 n)\right)\Exp[(Y^{\alpha_1,\alpha_2}_n)^2]\;,$$
which is too crude for the second moment argument to work. 

Secondly, the notion of a ``set of vertices following a certain histogram'' makes sense only in the stochastic block model, but not when we have a finite set of vertices in an uncountable probability space. Let us jump ahead and note that in the rigorous proof in Section~\ref{sec:mainproof} we, in a sense, are able to make use of histograms in the continuous setting. Namely, Lemma~\ref{lem:secondsamplinglemma} allows us to discretize a given graphon in an appropriate sense, after which it does make sense to talk about histograms. 

Let us remark that for the stochastic block model the first issue (which is the only in that case) can be dealt with by pedestrian calculations, thus yielding a routine proof of Theorem~\ref{thm:main} for the special class of stochastic block models.

\section{Tools for the proof of Theorem~\ref{thm:main}}\label{sec:tools}
In this section we prepare tools for the lower bound in Theorem~\ref{thm:main}.  In Section~\ref{ssec:maxl1} we state and prove Lemma~\ref{lem:maxl1} which asserts that if $f^*$ is a histogram almost attaining the supremum in~\eqref{eq:defKappa} then $\Gamma(f,W)$ is almost positive for all subhistograms $f\le f^*$. The need for this lemma was motivated in Section~\ref{ssec:IntroSecondMoment}. In Section~\ref{ssec:xi} we introduce a new graphon parameter $\xi(W)$. This parameter is motivated by controlling the second moment of the number of cliques of a given size. All the work in Section~\ref{sec:tools} steers towards deriving the two main results of this section, Lemma~\ref{lem:zoom} and Lemma~\ref{lem:firstandsecond}. The former lemma asserts that each graphon $W$ contains a subgraphon  $U$ with $\xi(U)\approx\frac{1}{\kappa(W)}$. The latter asserts that $\omega(\RG(n,U))\gtrapprox \frac{1}{\xi(U)}\cdot\log n$. These two lemmas combine easily to give the proof of the lower bound in Theorem~\ref{thm:main} (as is shown in Section~\ref{sec:mainproof}).

\subsection{Subhistograms of optimal histograms are admissible}\label{ssec:maxl1}
The main result of this section, Lemma~\ref{lem:maxl1}, tells us that if $f^*$ is a histogram almost attaining the supremum in~\eqref{eq:defKappa} then $\Gamma(f,W)$ is almost positive for all subhistograms $f\le f^*$. We showed that this particular case of the (false, in general) Dream Lemma is needed for the second moment to work. The proof of Lemma~\ref{lem:maxl1} is technical, building on Lemma~\ref{cl:AndrasClaim}. It turns out that in those situations when the supremum in~\eqref{eq:defKappa} is attained, Lemma~\ref{lem:maxl1} has a much shorter (but conceptually the same) proof. We offer this simplified statement in Lemma~\ref{lem:maxl1App} in the Appendix.

	\begin{lemma}\label{cl:AndrasClaim}
		Suppose that $W$ is an arbitrary  graphon with $0<\essinf W\le\esssup W<1$.
		Then there is a constant $K>0$ depending only on the graphon $W$ such that the following holds:
		Let $g$ be an arbitrary histogram admissible for $W$ and let $\delta\in(0,1)$.
		Suppose that $a\in (0,1)$ and that $g=g'+g''$ for some non-trivial histograms $g'$ and $g''$ such that $\|g'\|_1<\|g\|_1-\delta$. 
		Then either $\Gamma(g')\ge -a$, or there exists a histogram $g^*$ which is admissible for $W$ and for which we have
		\begin{equation}\label{eq:biggerg}
		\|g^*\|_1\ge \|g\|_1+K\delta^3a^{\frac52}\;.
		\end{equation}
	\end{lemma}
	\begin{proof}
		Since we shall work exclusively with the graphon $W$, we write $\Gamma(\cdot)$ as a shortcut for $\Gamma(\cdot,W)$. Let us write $m^-=\essinf W$ and $m^+=\esssup W$.
		
		Let us fix numbers $a,\delta\in (0,1)$ and a decomposition $g=g'+g''$ of $g$ into non-trivial histograms $g',g''$ such that $\|g'\|_1<\|g\|_1-\delta$.
		For $\epsilon_1\in(0,1)$ and $\epsilon_2>0$, let us write $g^*(\eps_1,\eps_2)=(1-\epsilon_1)g'+(1+\epsilon_2)g''$.
		Let us also write
		\begin{align}
		\begin{split}\label{eq:ABCDE}
		A&=\|g'\|_1\;, \\  
		B&=\|g''\|_1\;,\\
		C&=-\frac{1}{2}\int_{(x,y)\in\Omega^2} g'(x)g'(y)\log W(x,y)\;\differential(\nu^2)\;,\\ D&=-\frac{1}{2}\int_{(x,y)\in\Omega^2} g''(x)g''(y)\log W(x,y)\;\differential(\nu^2), \\
		 E&=-\int_{(x,y)\in\Omega^2} g'(x)g''(y)\log W(x,y)\;\differential(\nu^2)\;,
		\end{split}
		\end{align}
		and note that $A,B,C,D,E> 0$.
		We have $\Gamma(g')=A-C$ and $\Gamma(g)=A+B-C-D-E$.  
		There is nothing to prove when $\Gamma(g')\ge -a$. Thus, assume otherwise. Then
		\begin{equation}\label{eq:ACx}
		A<C-a\;.
		\end{equation}
		Upper-bounding $C$ by $\tfrac 12\|g'\|_1^2\log(1/m^-)$ and using that $C>a$, we get
		\begin{equation}\label{eq:g'largenorm}
		\|g'\|_1>\sqrt{\frac{2a}{\log\left(\nicefrac{1}{m^-}\right)}}\;.
		\end{equation}
		For each $\epsilon_1\in(0,1)$ and $\epsilon_2>0$, the difference $\Gamma(g^*(\eps_1,\eps_2))-\Gamma(g)$ can be expressed as
		$$(1-\eps_1)A + (1+\eps_2)B - (1-\eps_1)^2C - (1+\eps_2)^2D - (1-\eps_1)(1+\eps_2)E - \left( A + B - C - D - E \right)$$
		$$= \eps_1 ( -A + 2C + E ) + \eps_2( B - 2D -E ) - \eps_1^2C - \eps_2^2D + \eps_1\eps_2E.$$
		In particular, if $\eps_2=(1+\beta)\frac{A}{B}\eps_1$ (where $\epsilon_1\in (0,1)$ and $\beta>0$ will be determined later) then we have
		\begin{align}
		\begin{split}
		\label{choice of eps2}
		&\Gamma\left(g^*(\eps_1,(1+\beta)\tfrac{A}{B}\eps_1)\right)-\Gamma(g) \\
		= 
		&\eps_1\left(2C+E-\tfrac{2AD}{B}-\tfrac{AE}{B}\right) + \eps_1\cdot\eps_1\left(-C-\tfrac{A^2D}{B^2}+\tfrac{AE}{B}-2\tfrac{A^2}{B^2}\beta D+\beta\tfrac{AE}{B}-\tfrac{A^2}{B^2}\beta^2D\right) + \epsilon_1\beta(A-\tfrac{2AD}{B}-\tfrac{AE}{B})\\
		\ge 
		&\eps_1\underbrace{\left(2C+E-\tfrac{2AD}{B}-\tfrac{AE}{B}\right)}_{\texttt{T1}} - \eps_1\cdot\underbrace{\eps_1\left(C+\tfrac{A^2D}{B^2}+2\tfrac{A^2}{B^2}\beta D+\tfrac{A^2}{B^2}\beta^2D\right)}_{\texttt{T2}} - \epsilon_1\underbrace{\beta(\tfrac{2AD}{B}+\tfrac{AE}{B})}_{\texttt{T3}}\;.
		\end{split}
		\end{align}
		Let us expand the term \texttt{T1}.
		\begin{equation*}
		\begin{split}
		2C+E-\frac{2AD}{B}-\frac{AE}{B}&\gByRef{eq:ACx}2A-\frac{2AD}{B}-\frac{AE}{B}+2a\\
		&>2\frac{A}{B} ( B - D - E)+2a\\
		&>2\frac{A}{B}\left(B-D-E+(A-C)\right)+2a\\
		&=2\frac{A}{B}\Gamma(g)+2a\\
		&\ge 2a.
		\end{split}
		\end{equation*}
		Now, set $\eps_1=\frac{a}{4}\min(\tfrac{1}{C},\tfrac{B^2}{2A^2D})$ and $\beta=\min(1,\tfrac{aB}{4AD},\tfrac{aB}{2AE})$. Routine calculations give that each of the terms \texttt{T2} and \texttt{T3} is smaller than $a$. Plugging the bounds above in~\eqref{choice of eps2}, we get
		\begin{equation}
		\label{unif positive}
		\Gamma\left(g^*(\eps_1,(1+\beta)\tfrac{A}{B}\eps_1)\right)\ge \Gamma(g)\ge 0.
		\end{equation}
		We have
		\begin{equation*}
		C=\frac{1}{2}\int_{(x,y)\in\Omega^2} g'(x)g'(y)\log \left(\nicefrac{1}{W(x,y)}\right)\;\differential(\nu^2) \le
		\frac{1}{2}\int_{(x,y)\in\Omega^2} g(x)g(y)\log\left(\nicefrac{1}{m^-}\right) \;\differential(\nu^2)\le
		\frac{1}{2}\kappa(W)^2\log \frac{1}{m^-}
		\end{equation*}
		and similarly $D \le \frac{1}{2}\kappa(W)^2\log\left(\nicefrac{1}{m^-}\right)$ and $E \le \kappa(W)^2\log\left(\nicefrac{1}{m^-}\right)$.\Referee{9} We also have
		\begin{equation*}
		\frac{B}{A} =
		\frac{\|g-g'\|_1}{\|g'\|_1} \ge
		\frac{\delta}{\|g\|_1-\delta} \ge
		\frac{\delta}{\kappa(W)-\delta}\ge\frac{\delta}{\kappa(W)}\;.
		\end{equation*}
		Therefore
		\begin{align}\label{eps_1beta}
		\begin{split}
		\eps_1\beta&\ge\min\left(\frac{a}{4C},\frac{aB^2}{8A^2D},\frac{a^2B}{16ACD},\frac{a^2B^3}{32A^3D^2},\frac{a^2B}{8ACE},\frac{a^2B^3}{16A^3DE}\right)\\
		&\ge a^2\min\left(\frac{1}{4C},\frac{B^2}{8A^2D},\frac{B}{16ACD},\frac{B^3}{32A^3D^2},\frac{B}{8ACE},\frac{B^3}{16A^3DE}\right)\\
		&\ge
		a^2\frac{1}{\log^2 \left(\nicefrac{1}{m^-}\right)}\min\left(
		\frac{\log \left(\nicefrac{1}{m^-}\right)}{2\kappa(W)^2},\frac{\delta^2\log \left(\nicefrac{1}{m^-}\right)}{4\kappa(W)^4},\frac{\delta}{4\kappa(W)^5},\frac{\delta^3}{8\kappa(W)^7},\frac{\delta}{4\kappa(W)^5},\frac{\delta^3}{8\kappa(W)^7}\right)\\
		&\ge a^2\delta^3\frac{1}{\log^2 \left(\nicefrac{1}{m^-}\right)}\min\left(
		\frac{\log \left(\nicefrac{1}{m^-}\right)}{2\kappa(W)^2},\frac{\log \left(\nicefrac{1}{m^-}\right)}{4\kappa(W)^4},\frac{1}{4\kappa(W)^5},\frac{1}{8\kappa(W)^7},\frac{1}{4\kappa(W)^5},\frac{1}{8\kappa(W)^7}\right)\;.
		\end{split}
		\end{align}
		It follows that
		\begin{align*}
		&\|g^*\left(\eps_1,\eps_2\right)\|_1 = (1-\eps_1)A + (1+(1+\beta)\tfrac{A}{B}\eps_1)B=\|g\|_1+\beta\epsilon_1 \|g'\|_1\\ {\stackrel{{(\ref{eq:g'largenorm}), (\ref{eps_1beta})}}{>}}& \|g\|_1+a^{\frac 52}\delta^3\min\left(
		\frac{\log \left(\nicefrac{1}{m^-}\right)}{2\kappa(W)^2},\frac{\log \left(\nicefrac{1}{m^-}\right)}{4\kappa(W)^4},\frac{1}{4\kappa(W)^5},\frac{1}{8\kappa(W)^7},\frac{1}{4\kappa(W)^5},\frac{1}{8\kappa(W)^7}\right)\sqrt{\frac{2}{\log^5\left(\nicefrac{1}{m^-}\right)}}.
		\end{align*}
		This finishes the proof.
	\end{proof}

\begin{lemma}\label{lem:maxl1}
	Suppose that $W$ is an arbitrary  graphon with $0<\essinf W\le\esssup W<1$. Then there exists a number $\gamma_0>0$ and a function $q:(0,\gamma_0)\rightarrow \RR_+$ with $\lim_{\gamma\searrow0}q(\gamma)=0$ such that the following holds.
	Let $f^*$ be an admissible histogram for $W$ and let $\gamma\in(0,\gamma_0)$. Suppose that $\|f^*\|_1\ge \kappa(W)-\gamma$. Then for every $f\in \BOX(f^*)$ we have $\Gamma(f,W)\ge -q(\gamma)$.
\end{lemma}
\begin{proof}
	Let $K$ be the constant from Lemma~\ref{cl:AndrasClaim}. Set $\gamma_0$ so that $\gamma_0<K^2$. Suppose that $\gamma\in(0,\gamma_0)$ and $f^*$ is an admissible histogram with $\|f^*\|_1\ge \kappa(W)-\gamma$.
	
	We write $\Gamma(\cdot)$ as a shortcut for $\Gamma(\cdot,W)$.
	
	Let $f\in \BOX(f^*)$ be non-trivial. Suppose first that $\|f\|_1\ge \|f^*\|_1-\sqrt[9]{\gamma}$. Using~\eqref{eq:A} we get 
	\begin{equation}\label{eq:Gamma1}
		\begin{split}
			\Gamma(f)&=\int_{x\in\Omega} f(x)\;\differential(\nu)+\frac12\int_{(x,y)\in\Omega^2} f(x)f(y)\log W(x,y)\;\differential(\nu^2)\\
			&\ge\int_{x\in\Omega} f^*(x)\;\differential(\nu)-\sqrt[9]{\gamma}+\frac12\int_{(x,y)\in\Omega^2} f^*(x)f^*(y)\log W(x,y)\;\differential(\nu^2)\\
			&= \Gamma(f^*)-\sqrt[9]{\gamma}\ge -\sqrt[9]{\gamma}\;.
		\end{split}
	\end{equation}
	Suppose next that $\|f\|_1< \|f^*\|_1-\sqrt[9]{\gamma}$. We apply Lemma~\ref{cl:AndrasClaim} to
	$$g=f^*,\ \ \ g'=f,\ \ \ g''=f^*-f,\ \ \ \delta=\sqrt[9]{\gamma},\ \ \ a=\frac{\sqrt[5]{\gamma}}{K^{\frac{2}{5}}}\;.$$
	Then there is no histogram $g^*$ admissible for $W$ such that $\|g^*\|_1\ge \|f^*\|_1+K\delta^3a^{\frac52}$ as the right hand side equals $\|f^*\|_1+\gamma^{\frac{5}{6}}>\|f^*\|_1+\gamma\ge\kappa(W)$.
	So the lemma tells us that
	\begin{equation}\label{eq:Gamma2}
		\Gamma(f)\ge -\frac{\sqrt[5]{\gamma}}{K^{\frac{2}{5}}}\;.
	\end{equation}
	
	Combining \eqref{eq:Gamma1} with~\eqref{eq:Gamma2}, it suffices to define the function $q$ by $$q(\gamma)=\max\left(\sqrt[9]{\gamma},\frac{\sqrt[5]{\gamma}}{K^{\frac{2}{5}}}\right)\;,\ \ \ q>0\;,$$
	since then it is clear that $\lim_{\gamma\searrow0}q(\gamma)=0$.
\end{proof}

\subsection{The graphon parameter $\xi(\cdot)$}\label{ssec:xi}
In Section~\ref{ssec:IntroSecondMoment} we outlined why the second moment argument for counting cliques should go through. (Recall that the second moment argument is needed to prove the lower bound in Theorem~\ref{thm:main}, which is the more difficult half of the statement.) For the actual execution of this step, however, we need to introduce a new graphon parameter. This parameter is a version of the cut norm with an exotic scaling. Given an arbitrary graphon $W$ represented on a probability space $\Omega$ we define
\begin{equation}\label{eq:defxi}
\xi(W)=\sup_{B\subset \Omega,\nu(B)>0} \frac1{2\nu(B)}\int_{(x,y)\in B\times B} \log (\nicefrac{1}{W(x,y)})\;\differential(\nu^2)\;.
\end{equation}\Referee{10}
The key feature of $\xi(W)$, which we prove in Lemma~\ref{lem:firstandsecond}, is that the second moment argument for counting cliques of order almost $\frac{1}{\xi(W)}\log n$ works. More precisely, in the proof of Lemma~\ref{lem:firstandsecond}, we set up a random variable $Y$ which essentially counts the number of individually-weighted cliques of the said order.\footnote{The weighting of the particular cliques is a technical but important subtlety, see~\eqref{eq:modM} below.} We show that $\Exp[Y^2]\approx\Exp^2[Y]$. This allows us to conclude that there must be at least one clique of such an order.

Of course, Lemma~\ref{lem:firstandsecond} itself is not enough to establish the lower bound in Theorem~\ref{thm:main}: we need to connect the new quantity $\xi(W)$ to the original quantity $\kappa(W)$. Given Lemma~\ref{lem:firstandsecond} described above, we would hope that $\kappa(W)=\frac1{\xi(W)}$. Unfortunately, in general, we only have $\kappa(W)\ge\frac1{\xi(W)}$, see Fact~\ref{fact:xikappa}. Not all is lost though. In Lemma~\ref{lem:zoom} we prove that every graphon $W$ contains a subgraphon $U$ for which $\frac{1}{\xi(U)}> \kappa(W)-\epsilon$ (here, $\epsilon>0$ is arbitrarily small). After picking a subgraphon $U$ for which $\frac{1}{\xi(U)}\approx \kappa(W)$ we continue with the proof of the lower bound in Theorem~\ref{thm:main} as follows. We prove in Lemma~\ref{lem:firstandsecond} that asymptotically almost surely $\omega(\RG(n,U))\gtrapprox \frac{1}{\xi(U)}\log n$. As described in~\eqref{eq:PsiSubgraphon}, the above combination of Lemma~\ref{lem:zoom} and Lemma~\ref{lem:firstandsecond} will conclude the desired proof of the lower bound in Theorem~\ref{thm:main}.

\medskip

Now, let us state and prove the already advertised Fact~\ref{fact:xikappa}. This fact will not be needed in our proof of Theorem~\ref{thm:main}. However, since it is so basic we record it here. 
\begin{fact}\label{fact:xikappa}
	Let $W$ be an arbitrary graphon on a probability space $\Omega$, and let $U=W\restriction_{A\times A}$ be a subgraphon obtained by restricting $W$ to a set $A$ of positive measure. Then $\frac{1}{\xi(U)}\le \kappa(W)$.
\end{fact}
\begin{proof}
	By considering the set $B=A$ in~\eqref{eq:defxi} we see that $$\xi(U)\ge \frac{1}{2}\int_{(x,y)\in A^2}\log (\nicefrac{1}{W(x,y)})\;\differential(\nu_A^2)\eqByRef{eq:rescale}\frac{1}2\frac1{{\nu(A)^2}}\int_{(x,y)\in A^2}\log(\nicefrac{1}{W(x,y)})\;\differential(\nu^2)\geByRef{eq:kappa2} \frac{1}{\kappa(W)}\;.$$
\end{proof}

\medskip

In the rest of this section we prove Lemmas~\ref{lem:firstandsecond} and~\ref{lem:zoom}. The paths towards these lemmas are shown in Figure~\ref{fig:towardslemmas7}.
\begin{figure}[t]
    \centering
    \includegraphics{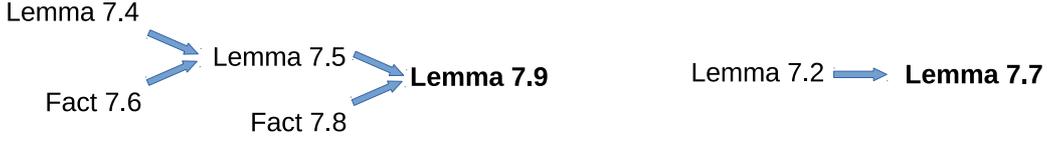}
     \caption{The scheme of the proofs of Lemmas~\ref{lem:firstandsecond} and~\ref{lem:zoom}.}
\label{fig:towardslemmas7}
\end{figure}

For the next lemma, note that if $G$ is a finite graph then the value of $\xi(W_G)$ does not depend on the particular representation $W_G$ of $G$.
\begin{lemma}\label{lem:differencegraphons}
Suppose that $c\in(0,1]$. Let $W$ be a graphon with $\essinf W\ge c$ and $G$ an edge-weighted complete graph with all edge-weights $w(i,j)$ in the interval $[c,1]$. Consider the ``negative logarithms of $W$ and $G$'', that is, an $L^\infty$-graphon $W'(x,y):=\log(\nicefrac{1}{W(x,y)})$ and a weighted graph $G'$ with $V(G')=V(G)$ and weight function $w'(i,j)=\log(\nicefrac{1}{w(i,j)})$. Then for an arbitrary $\gamma\in(0,1]$ we have
$$|\xi(W)-\xi(W_G)|\le \max\left(\tfrac\gamma2\log\tfrac1c,\tfrac1\gamma\delta_\square(W',G')\right)\;.$$
\end{lemma}
\begin{proof}
We shall prove the upper bound only for $\xi(W)-\xi(W_G)$. The upper bound on $\xi(W_G)-\xi(W)$ is done completely analogously. Suppose that $W$ is represented on a probability space $\Omega$. Looking at definition~\eqref{eq:defxi}, we need to provide an upper bound
\begin{equation}\label{eq:nowindow}
\underbrace{\frac1{\nu(A)}\int_{(x,y)\in A^2, x<y}\log \left(\nicefrac{1}{W(x,y)}\right)\;\differential(\nu^2)}_{\texttt{S1}} -
\underbrace{
\vphantom{\frac1{\nu(A)}\int_{(x,y)\in A^2, x<y}}
\xi(W_G)}_{\texttt{S2}}
\end{equation}
for each set $A\subset \Omega$ of positive measure. If $\nu(A)\le \gamma$ then the integral is over a set of measure at most $\frac12\nu^2(A)\le\frac\gamma2\nu(A)$. Thus, the term \texttt{S1} can be bounded from above by $-\frac\gamma2\log(\essinf W)\le\frac\gamma2 \log\frac1c$, as needed. 

Suppose next that $\nu(A)>\gamma$. Suppose first that $\delta_\square(W',G')=0$. Using the invertible measure preserving maps from~\eqref{eq:defcutdist}, we know that for each $\epsilon>0$ there exists a graphon representation $W_{G'}$ of $G'$ on $\Omega$ such that $d_\square(W',W_{G'})<\epsilon\nu(A)$. Then
\begin{align*}
&\frac1{\nu(A)}\int_{(x,y)\in A^2, x<y}\log \left(\nicefrac{1}{W(x,y)}\right)\;\differential(\nu^2) -
\frac1{\nu(A)}\int_{(x,y)\in A^2, x<y}{W_{G'}(x,y)}\;\differential(\nu^2)
\\
&=
\frac1{2\nu(A)}\left(\int_{(x,y)\in A^2}\left(W'(x,y)-W_{G'}(x,y)\right)\;\differential(\nu^2)\right)
\\
&\le \frac1{\nu(A)}\cdot d_\square(W',W_{G'})
\le \epsilon\;,
\end{align*}
as was needed.

Suppose next that  $\delta_\square(W',G')>0$. Using the invertible measure preserving maps from~\eqref{eq:defcutdist}, we know that for each $\epsilon>0$ there exists a graphon representation $W_{G'}$ of $G'$ on $\Omega$ such that $d_\square(W',W_{G'})<(1+\epsilon)
\delta_\square(W',W_{G'})$.
We shall fix such a representation $W_{G'}$ for $\epsilon=\frac{\nu(A)}{\gamma}-1$. Then~\eqref{eq:nowindow} can be bounded from above by
\begin{align*}
&\frac1{\nu(A)}\int_{(x,y)\in A^2, x<y}\log \tfrac1{W(x,y)} \;\differential(\nu^2)-
\frac1{\nu(A)}\int_{(x,y)\in A^2, x<y}{W_{G'}(x,y)}\;\differential(\nu^2)
\\
&=
\frac1{2\nu(A)}\left(\int_{(x,y)\in A^2}\left(W'(x,y)-W_{G'}(x,y)\right)\;\differential(\nu^2)\right)
\\
&\le \frac1{\nu(A)}\cdot d_\square(W',W_{G'})
\le \frac1\gamma\cdot\delta_\square(W',W_{G'})\;,
\end{align*}
as was needed.
\end{proof}
\begin{lemma}\label{lem:xismallsubgraphs}
Suppose that $c\in(0,1]$. Let $W$ be a graphon with $\essinf W\ge c$. Suppose that a sequence of integers $(k_n)_{n=1}^\infty$ has the property that $\sqrt{\log n}\le k_n\le \sqrt[3]{n}$. Suppose that $\epsilon>0$ is arbitrary.

In the weighted random graph $G\sim\RH(n,W)$ consider the family $\mathcal H$ of all sets $X\subset V(G)$ of size $k_n$ which have the property that $|\xi(W)-\xi(W_{G[X]})|\ge \epsilon$. Then asymptotically almost surely (as $n\to +\infty$) we have that $|\mathcal H|\le \epsilon{n\choose k_n}$.
\end{lemma}
Actually, the assertion of Lemma~\ref{lem:xismallsubgraphs} is violated only with probability at most $\exp(-\frac{n}{2\log n})$, as can be seen from the proof of Lemma~\ref{lem:xismallsubgraphs}. We shall not need this refinement, though.\Referee{11} For the proof of Lemma~\ref{lem:xismallsubgraphs} we shall need the following well-known fact which we include here for the reader's convenience.
\begin{fact}\label{fact:ballsbin}
Let us place $m$ balls independently at random into one of $n$ bins. If $n\ge m^3$ then with probability at least $1-2n^{-1/3}$ each bin contains at most one ball.
\end{fact}
\begin{proof}
Let us first bound the probability that one distinguished ball is placed into a bin which contains some other balls. Recall that for each $n\ge 2$,
\begin{equation}\label{eq:TaylorE}
1-\tfrac1n\ge\exp(-\tfrac2n)\;.
\end{equation}
The mentioned probability is exactly
$$1-(1-\tfrac1n)^{m-1}\leByRef{eq:TaylorE} 1-\exp(-\tfrac{2(m-1)}{n})\le 1-\exp(-\tfrac{2m}{n})\le 1-\exp(-2n^{-2/3})\le 2n^{-2/3}\;.$$
The claim then follows by summing this error probability over all $m\le n^{1/3}$ balls.	
\end{proof}
\begin{proof}[Proof of Lemma~\ref{lem:xismallsubgraphs}]
Let $\Omega$ be the probability space underlying $W$. Let $W'=\log\nicefrac{1}{W}$ be the negative logarithm of $W$. Note that $W'$ is bounded from above by $\log\nicefrac1c$. Sampling the random graph $G\sim \RH(n,W)$ can be naturally coupled with sampling a random graph $G'\sim \RH(n,W')$. So, for the first part of the argument, we shall analyze the graph $G'$.

Suppose first that $n$ is fixed. Corollary~\ref{cor:secondsamplingunbounded} implies that with probability at least $1-\exp(-\frac{n}{2\log n})=1-o(1)$ we have $\delta_\square(G',W')\le \frac{20\log\nicefrac1c}{\sqrt{\log n}}$. We shall prove the statement for each weighted graph $G'$ satisfying this property (provided that $n$ is sufficiently large). That means that we assume that $G'$ is fixed, and $G$ is its exponentiated version. In particular, all the probabilistic calculations below are only with respect to later randomized steps. Let $\mathcal K$ be the family of all subsets $X$ of $V(G')=V(G)$ of size $k_n$ for which $\delta_\square(G',G'[X])\ge \frac{20\log\nicefrac1c}{\sqrt{\log k_n}}$.

Consider the graphon representation $W_{G'}$ of $G'$ represented on a partition $A_1\dcup A_2\dcup\ldots \dcup A_n=\Omega$. Sample the graph $H\sim\RH(k_n,W_{G'})$. If we condition on the event $\mathcal E$ that the $k_n$ representatives of the vertices of $H$ in the sampling procedure were selected from pairwise distinct ``bins'' $A_1\dcup A_2\dcup\ldots \dcup A_n$ then $H$ is a uniformly random subgraph of $G'$ of order $k_n$. Fact~\ref{fact:ballsbin} gives that $\Pr[\mathcal E]\ge 1-2n^{-1/3}$. Thus,
$$\Pr\left[\delta_\square(H,W_{G'})\ge\tfrac{20\log\nicefrac1c}{\sqrt{\log k_n}}\right]\ge \Pr\left[\;\delta_\square(H,W_{G'})\ge\tfrac{20\log\nicefrac1c}{\sqrt{\log k_n}}\;|\;\mathcal E\right]\Pr[\mathcal E]\ge \frac{|\mathcal K|}{{n\choose k_n}}(1-2n^{-1/3})\;.$$
Another application of Corollary~\ref{cor:secondsamplingunbounded} gives that $\Pr[\delta_\square(H,W_{G'})\ge\tfrac{20\log\nicefrac1c}{\sqrt{\log k_n}}]<\exp(-\frac{k_n}{2\log k_n})$. Thus, we get (for $n$ sufficiently large) that $|\mathcal K|\le \epsilon {n\choose k_n}$. So, the lemma will follow provided that we prove that $\mathcal H\subset \mathcal K$, which we prove next.

Indeed, let $X\not \in \mathcal K$ be an arbitrary vertex set of size $k_n$. Then $\delta_\square(G'[X],W')\le \delta_\square(G'[X],G')+\delta_\square(W_{G'},W')\le \frac{20\log\nicefrac1c}{\sqrt{\log k_n}}+\frac{20\log\nicefrac1c}{\sqrt{\log n}}\le \frac{40\log\nicefrac1c}{\sqrt{\log k_n}}$. Then Lemma~\ref{lem:differencegraphons} tells us that for each $\gamma\in(0,1)$,
$$|\xi(W)-\xi(W_{G[X]})|\le \max\left(\tfrac{\gamma}2\log\nicefrac1c,\frac{1}\gamma\cdot\frac{40\log\nicefrac1c}{\sqrt{\log k_n}}\right)\;.$$
We take $\gamma=\nicefrac{1}{\sqrt[4]{\log k_n}}$, and see that the right-hand side is, for large enough $n$, smaller than $\epsilon$. This proves that $X\not\in\mathcal H$ and consequently concludes the lemma.
\end{proof}

Our next two lemmas are crucial in proving the lower bound in Theorem~\ref{thm:main}. The first lemma, Lemma~\ref{lem:zoom}, tells us that in every graphon $W$ there exists a subgraphon $U$ of $W$ for which we have $\frac{1}{\xi(U)}\gtrapprox\kappa(W)$.  The second lemma, Lemma~\ref{lem:firstandsecond}, then tells us that in $\RG(n,U)$ we can typically find cliques of order almost $\frac{1}{\xi(U)}\log n$. 

\begin{lemma}\label{lem:zoom}
	Suppose that $W:\Omega^2\rightarrow[0,1]$ is a graphon with $0<\essinf W\le\esssup W<1$. Then for every $\epsilon>0$ there exists a set $A\subset \Omega$ of positive measure such that for the subgraphon $U=W\restriction_{A\times A}$ we have $
	\frac1{\xi(U)}\ge \kappa(W)-\epsilon$.
\end{lemma}
\begin{proof}
	Let us write $m=\log\left(\nicefrac{1}{\essinf W}\right)$. Consider the number $\gamma_0>0$ and the function $q:(0,\gamma_0)\rightarrow \mathbb R_+$ given by Lemma~\ref{lem:maxl1} for the graphon $W$.
	
	Let $\delta>0$ be fixed such that $\delta \kappa(W)<\gamma_0$. We use~\eqref{eq:kappa2} to find a set $A$ of positive measure such that
	\begin{equation}\label{eq:chooseA}
		(1+\delta)\frac{2}{\kappa(W)}\ge 
		\frac{1}{\nu(A)^2}\int_{(x,y)\in A^2}\log(\nicefrac{1}{W(x,y)})\;\differential(\nu^2)\;.
	\end{equation}
	Consider now the subgraphon $U=W\restriction_{A\times A}$ on the probability space $A$ endowed with the measure $\nu_A$. 
	
	We now turn to obtaining the bound on $\xi(U)$. To this end we want to control each term in~\eqref{eq:defxi}.
	\begin{claim}\label{cl:xi}
		Suppose that $B\subset A$ is an arbitrary set of positive measure. We have 
		\begin{equation}\label{eq:Cl261}
			\frac{1}{1+\delta}\cdot\frac1{\nu_A(B)}\int_{(x,y)\in B^2}\log(\nicefrac{1}{W(x,y)})\differential(\nu_A^2)\le \frac2{\kappa(W)}+\sqrt{\frac{q(\delta\kappa(W))}{(1-\delta)\kappa(W)}}\cdot m\;.
		\end{equation}
	\end{claim}
	\begin{proof}[Proof of Claim~\ref{cl:xi}] In the following, we abbreviate $q=q(\delta\kappa(W))$.
		Suppose first that $$\nu_A(B)< \sqrt{\frac{q}{(1-\delta)\kappa(W)}}\;.$$ Then $$\frac1{\nu_A(B)}\int_{(x,y)\in B^2}\log (\nicefrac{1}{W(x,y)})\;\differential(\nu_A^2)\le\nu_A(B) m<\sqrt{\frac{q}{(1-\delta)\kappa(W)}}\cdot m\;,$$
		as needed.
		
		So, it remains to consider the case
		\begin{equation}\label{secondCase}
			\nu_A(B)\ge \sqrt{\frac{q}{(1-\delta)\kappa(W)}}\;.
			\end{equation}
			Let $c>0$ be maximum such that $\Gamma(c\cdot\mathbf{1}_A,W)\ge0$. That is, we have $\Gamma(c\cdot \mathbf{1}_A,W)=0$ which can be rewritten using~\eqref{eq:A} as\Referee{13} $0=c\nu(A)+c^2\frac12\int_{(x,y)\in A^2}\log W(x,y) \;\differential(\nu^2)$. Thus,
		\begin{equation}\label{eq:cequal}
			c=\frac{2\nu(A)}{\int_{(x,y)\in A^2}\log(\nicefrac{1}{W(x,y)})\;\differential(\nu^2)}\;.
		\end{equation}
		Now the assumption of Proposition~\ref{proposition:differentkappa} is satisfied by (\ref{eq:chooseA}) (after a change of $W$ on a null set).\footnote{Proposition~\ref{proposition:differentkappa} was formulated only for graphons represented on the unit interval. However, we can use Fact~\ref{fact:representinterval} to represent $W$ on the unit interval, and then we can apply Proposition~\ref{proposition:differentkappa}.}\Referee{14}
		The ``more precisely'' part of this proposition tells us that 
\begin{equation}\label{eq:cge}
c\ge(1-\delta)\frac{\kappa(W)}{\nu(A)}\;. 
\end{equation}
		
		Consider now the function $f=c\cdot\mathbf{1}_A$. We have $\|f\|_1\ge (1-\delta)\kappa(W)$. Thus Lemma~\ref{lem:maxl1} tells us that $\Gamma(g,W)\ge -q$ for each subhistogram $g$ of $f$. Let us apply this to the function $g:=c\cdot \mathbf{1}_B$.
		Then
		\begin{align*}
			\Gamma(g,W)=c\nu(B)-\frac12c^2\int_{(x,y)\in B^2}\log (\nicefrac{1}{W(x,y)})\;\differential(\nu^2)\ge -q \;,
		\end{align*}
		yielding
		\begin{align*}
			\nu(B)\ge \frac12 c \int_{(x,y)\in B^2}\log(\nicefrac{1}{W(x,y)})\;\differential(\nu^2)-\frac{q}{c}
			\overset{\eqref{eq:cequal}}{=}
			\nu(A)\frac{\int_{(x,y)\in B^2}\log (\nicefrac{1}{W(x,y)})\;\differential(\nu^2)}{\int_{(x,y)\in A^2}\log (\nicefrac{1}{W(x,y)})\;\differential(\nu^2)}
			-\frac{q}{c}\;.
		\end{align*}
		This can be rewritten as 
		\begin{align}
			\begin{split}
				\label{eq:zahrada}
				\frac{1}{\nu(A)}\int_{(x,y)\in A^2}\log (\nicefrac{1}{W(x,y)})\;\differential(\nu^2)
				&\ge \frac{1}{\nu(B)}\int_{(x,y)\in B^2}\log(\nicefrac{1}{W(x,y)})\;\differential(\nu^2)-\frac{q\int_{(x,y)\in A^2}\log(\nicefrac{1}{W(x,y)})\;\differential(\nu^2)}{c\nu(A)\nu(B)}\\
				&\ge
				\frac{1}{\nu(B)}\int_{(x,y)\in B^2}\log(\nicefrac{1}{W(x,y)})\;\differential(\nu^2)-\frac{qm\nu(A)}{c\nu(B)}\\
				&\geByRef{eq:cge}
				\frac{1}{\nu(B)}\int_{(x,y)\in B^2}\log(\nicefrac{1}{W(x,y)})\;\differential(\nu^2)-\frac{qm\nu(A)}{(1-\delta)\kappa(W)\nu_A(B)}\\
				&\stackrel{(\ref{secondCase})}{\ge}
				\frac{1}{\nu(B)}\int_{(x,y)\in B^2}\log(\nicefrac{1}{W(x,y)})\differential(\nu^2)-\sqrt{\frac{q}{(1-\delta)\kappa(W)}}\cdot m\cdot\nu(A)\;.
			\end{split}
		\end{align}
		Thus,
		\begin{align*}
			\frac{2}{\kappa(W)}&\geByRef{eq:chooseA} \frac{1}{1+\delta}
			\frac{1}{\nu(A)^2}\int_{(x,y)\in A^2}\log(\nicefrac{1}{W(x,y)})\;\differential(\nu^2)\\
			&\geByRef{eq:zahrada}
			\frac{1}{1+\delta}\frac{1}{\nu(A)\nu(B)}\int_{(x,y)\in B^2}\log(\nicefrac{1}{W(x,y)})\;\differential(\nu^2)-\frac{1}{1+\delta}\sqrt{\frac{q}{(1-\delta)\kappa(W)}}\cdot m\\
			&\geByRef{eq:rescale}
			\frac{1}{1+\delta}\frac{1}{\nu_A(B)}\int_{(x,y)\in B^2}\log(\nicefrac{1}{W(x,y)})\;\differential(\nu_A^2)-\sqrt{\frac{q}{(1-\delta)\kappa(W)}}\cdot m\;,
		\end{align*}
		as required.
	\end{proof}
	The term $\sqrt{\frac{q(\delta\kappa(W))}{(1-\delta)\kappa(W)}}$ in~\eqref{eq:Cl261} does not depend on the choice of the set $A$. Thus, it tends to zero as we let $\delta\searrow 0$. We conclude that for $\delta>0$ sufficiently small, if we select~$A$ as in~\eqref{eq:chooseA}, we have $$\frac1{2\nu_A(B)}\int_{(x,y)\in B^2}\log(\nicefrac{1}{W(x,y)})\differential(\nu_A^2)\le \frac{1+\delta}{\kappa(W)}+\frac{1+\delta}{2}\cdot\sqrt{\frac{q(\delta\kappa(W))}{(1-\delta)\kappa(W)}}\cdot m$$ for each $B\subset A$ of positive measure. By~\eqref{eq:defxi}, we have $$\xi(U)\le\frac{1+\delta}{\kappa(W)}+\frac{1+\delta}2\cdot\sqrt{\frac{q(\delta\kappa(W))}{(1-\delta)\kappa(W)}}\cdot m\;.$$ If $\delta>0$ is sufficiently small then the right hand side (which tends to $\tfrac{1}{\kappa(W)}$ as $\delta\searrow 0$) is smaller than $\frac{1}{\kappa(W)-\eps}$, and so $$\frac{1}{\xi(U)}\ge\kappa(W)-\eps\;,$$
	as was needed.
\end{proof}

We shall need the following observation.
\begin{fact}\label{fac:productweight}
Let $G$ be a finite edge-weighted complete graph with vertex set $[n]$ whose symmetric weight function $w:V(G)^2\rightarrow [0,1]$ puts weight 1 on all self-loops. Then for every $C\subset [n]$ and for the graphon representation $W_G$ of $G$ we have $$\prod_{\substack{i,j\in C\\ i<j}} w(i,j)\ge\exp(-\xi(W_G)n|C|)\;.$$
\end{fact}
\begin{proof}
Suppose that $W_G$ is a representation of $G$ on the unit interval $I$. Suppose further that each vertex $i\in [n]$ is represented by an interval $D_i\subset I$ (c.f. Remark~\ref{rem:graphononinterval}), and that for each $1\le i<j\le n$ the interval $D_i$ lies to the left of the interval $D_j$. Of course, such a change of representation of $W_G$ does not change $\xi(W_G)$.\Referee{15}

The case $C=\emptyset$ is trivial, so assume $C\neq\emptyset$. Consider the set $X=\bigcup_{i\in C}D_i$. Definition~\eqref{eq:defxi} gives that
$$\xi(W_G)\ge \frac1{\lambda(X)}\int_{(x,y)\in X^2, x<y} \log(\nicefrac{1}{W_G(x,y)})\;\differential(\lambda^2).$$\Referee{15}
Let us split the integration above according to the partition $(D_i\times D_j)_{i,j\in [n]}$. We can neglect the terms for which $i=j$ since then the integrand is $\log(\nicefrac11)=0$. So, suppose that $i<j$. Then for each $x\in D_i$, $y\in D_j$ we have $W_G(x,y)=w(i,j)$. Thus, in this case $\int_{(x,y)\in D_i\times D_j,x<y} \log(\nicefrac{1}{W_G(x,y)})\;\differential(\lambda^2)=\frac{1}{n^2}\log(\nicefrac{1}{w(i,j)})$. We conclude that
$$\xi(W_G)\ge\frac{n}{|C|}\cdot\frac{1}{n^2}\sum_{\substack{i,j\in C\times C\\{i<j}}}\log(\nicefrac{1}{w(i,j)})\;$$
The lemma follows after exponentiation.
\end{proof}
\begin{lemma}\label{lem:firstandsecond}
Suppose that $W$ is a graphon with $\essinf W>0$. Suppose that $\alpha<1/\xi(W)$. Then asymptotically almost surely, $\RG(n,W)$ contains a clique of order $\alpha \log n$.
\end{lemma}
\begin{proof}
Choose $\delta>0$ such that $\alpha(\xi(W)+\delta)<1$ and let us write 
\begin{equation}\label{eq:setdelta}
\gamma=1-\alpha(\xi(W)+\delta)\;.
\end{equation}

Let $H\sim \RH(n,W)$. Set $k:=\alpha\log n$. Let $\mathcal A$ be be the family of all sets $A\subseteq V(H)$ of size $k$ for which $|\xi(W)-\xi(W_{H[A]})|<\delta$. Lemma~\ref{lem:xismallsubgraphs} tells us that with high probability, the graph $H$ has the property that $|\mathcal A|\ge (1-\delta){n\choose k}$. Condition on this event, and fix a realization of the weighted graph $H$ with a weight function $w:{V(H)\choose 2}\rightarrow [0,1]$ having the above property.

We shall now obtain from $H$ an unweighted graph $G$ by including each edge $ij$ with probability $w(i,j)$. It is our task to show that with high probability, $G$ contains a clique of order $k$. (Recall that this probability is only with respect to obtaining $G$ from $H$.)

For each $A\in\mathcal A$ set up the indicator $X_A$ of the event that $G[A]$ is a clique. Define 
\begin{equation}\label{eq:modM}
Y_A:=\frac{X_A}{\Exp[X_A]}
\end{equation}
(note that the denominator is not zero because $\essinf W>0$). Let $Y=\sum_{A\in\mathcal A}Y_A$. To conclude the proof, we want to prove that for each $\epsilon>0$ (which we now consider fixed), we have
\begin{equation}\label{eq:Ypositive}
\mbox{$Y>0$ with probability at least $1-\epsilon$,}
\end{equation}
provided that $n$ is sufficiently large.

We have $\Exp[Y_A]=1$ for each $A\in\mathcal A$, and consequently $\Exp[Y]=|\mathcal A|$, which tends to infinity with $n\to +\infty$. Below we shall prove that $\Exp[Y^2]\le(1+\epsilon)\Exp^2[Y]$, which will establish~\eqref{eq:Ypositive} via the usual second-moment argument.
\begin{claim}\label{cl:secmomY}
We have $\Exp[Y^2]<(1+\epsilon)\Exp^2[Y]$.
\end{claim}
\begin{proof}[Proof of Claim~\ref{cl:secmomY}]
For $\ell=0,1,\ldots,k$, let us write $$M_\ell=\sum_{\substack{A,B\in\mathcal A\\|A\cap B|=\ell}}\Exp[Y_AY_B]\;.$$
Then we have $\Exp[Y^2]=\sum_{\ell=0}^k M_\ell$. So, it is our goal to bound each of the numbers $M_\ell$. We have
\begin{equation}\label{eq:Mnull}
M_0\le {n\choose k}^2\;.
\end{equation}
For $\ell>0$ we have\footnote{the calculations below are also valid in the case $\ell=0$, but we shall not use them in that case}
\begin{align*}
M_\ell=\sum_{\substack{A,B\in\mathcal A\\|A\cap B|=\ell}}\Exp[Y_AY_B]=\sum_{C\in{V(G)\choose \ell}}\sum_{\substack{A,B\in\mathcal A\\A\cap B=C}}\frac{\Exp[X_AX_B]}{\Exp[X_A]\Exp[X_B]}\;.
\end{align*}
Given two sets $A,B\in\mathcal A$, it is easy to see that $$\Exp[X_AX_B]=\frac{\Exp[X_A]\Exp[X_B]}{\prod_{ij\in {A \cap B\choose 2}}w(i,j)}\;.$$ Thus,
\begin{align}
\begin{split}
\label{eq:Mellfirstline}
M_\ell&=\sum_{C\in{V(G)\choose \ell}}\sum_{\substack{A,B\in\mathcal A\\A\cap B=C}}\prod_{ij\in {C\choose 2}}w^{-1}(i,j)\\
\JUSTIFY{Fact~\ref{fac:productweight} applied on the graph $H[A]$ and subset $C$}&\le
\sum_{C\in{V(G)\choose \ell}}\sum_{\substack{A,B\in\mathcal A\\A\cap B=C}}\exp\left((\xi(W)+\delta)k\ell\right)
\end{split}\\
\nonumber
&\le {n\choose \ell \:|\:k-\ell\:|\:k-\ell}\exp\left((\xi(W)+\delta)k\ell \right)\\
\nonumber
&= \frac{n!(n-k)!^2}{(n-2k+\ell)!n!^2}\cdot
\frac{k!^2}{\ell!(k-\ell)!^2}\cdot
{n\choose k}^2
\exp\left((\xi(W)+\delta)k\ell \right)\\
\nonumber
\JUSTIFY{$n$ is sufficiently large}
&\le \left(\frac{2}{n}\right)^\ell\cdot
k^{2\ell}\cdot
{n\choose k}^2
\exp\left((\xi(W)+\delta)k\ell \right)\\
\nonumber
&= \left(\frac{2k^2\exp\left((\xi(W)+\delta)k \right)}{n}\right)^\ell\cdot{n\choose k}^2\\
&\leByRef{eq:setdelta}
\left(2^{\ell}k^{2\ell}\exp\left(-\gamma \ell\log n\right)\right)\cdot{n\choose k}^2
\nonumber
\\
\JUSTIFY{$n\gg k$}&\le
\exp\left(-\tfrac{\gamma \ell\log n}{2}\right)\cdot{n\choose k}^2\;.
\nonumber
\end{align}
Recall that $\Exp[Y]=|\mathcal A|\ge (1-\delta){n\choose k}$. Thus,\Referee{16}
\begin{align*}
\frac{\Exp[Y^2]}{\Exp^2[Y]}
&= \frac{\sum_{\ell=0}^{k}M_\ell}{\Exp^2[Y]}
\\
&\leByRef{eq:Mnull} \frac 1{(1-\delta)^2}+ \frac{\sum_{\ell=1}^{k}M_\ell}{(1-\delta)^2{n\choose k}^2}
\\
&\leByRef{eq:Mellfirstline}\frac 1{(1-\delta)^2}+\frac{1}{(1-\delta)^2}\sum_{\ell=1}^k\exp\left(-\tfrac{\gamma \ell\log n}{2}\right)\\
&\le\frac{1}{(1-\delta)^2}\sum_{\ell=0}^{\infty}\exp\left(-\tfrac{\gamma \ell\log n}{2}\right)\;.
\end{align*}
Note that the last expression is a geometric series, and its quotient $\exp\left(-\frac{\gamma \log n}{2}\right)$ tends to $0$ as $n\rightarrow\infty$. Therefore the sum of the series tends to $1$.
Thus for sufficiently large $n$ and for sufficiently small $\delta>0$ we get $\frac{\Exp[Y^2]}{\Exp^2[Y]}< 1+\epsilon$, as was needed.
\end{proof}
Claim~\ref{cl:secmomY} tells us that $\Var[Y]\le \epsilon\Exp^2[Y]$. Therefore,~\eqref{eq:Ypositive} follows from Chebyshev's Inequality.
\end{proof}

\section{Proof of Theorem~\ref{thm:main}}\label{sec:mainproof}
Let $c=\essinf W$. Suppose that $W$ is represented on the unit interval $I=(0,1)$ equipped with the Lebesgue measure $\lambda$. Let us replace the value of $W$ in every point $(x,y)\in(0,1)^2$ that is not a point of approximate continuity by $c$. This is a change of measure zero by Fact~\ref{theorem:appcont}. In particular, $\kappa(W)$ does not change, nor does the distribution of the model $\RG(n,W)$.

\subsection{Upper bound}
Let $\epsilon\in (0,\kappa(W)/4)$ be arbitrary. Let $n$ be sufficiently large. We want to show that a.a.s.~$G\sim\RG(n,W)$ contains no clique of order $k=(\kappa(W)+\epsilon)\log n$. Let $X_n(G)$ count such cliques. We have
$$ \Exp[X_n(G)]=\int_{(x_1,x_2,\ldots,x_n)\in I^n}\sum_{A\in{[n]\choose k}}\prod_{i,j\in A, i<j} W(x_i,x_j)\;\differential(\lambda^n)\;.
$$
This summation has ${n\choose k}<n^k=\exp\left((\kappa(W)+\epsilon)\log^2 n\right)$ terms.
By~\eqref{eq:kappa1}, each of these terms is bounded by $P_k^{\frac{k(k-1)}{2}}$ where $\lim\limits_{k\rightarrow\infty}P_k=\exp\left(-\frac{2}{\kappa(W)}\right)$. So if $n$ is sufficiently large then each term is bounded by
$$\exp\left(-\frac{2}{\kappa(W)}\cdot\frac{k(k-1)}{2}+\eps\right)\;.$$
Thus,
\begin{equation*}
\begin{split}
\Exp[X_n(G)]&\le \exp\left((\kappa(W)+\epsilon)\log^2 n-\frac{2}{\kappa(W)}\cdot\frac{k(k-1)}{2}+\eps\right)\\
&=
\exp\left(
\left(-\eps-\frac{\eps^2}{\kappa(W)}\right)\log^2n+\left(1+\frac{\eps}{\kappa(W)}\right)\log n+\eps
\right)
\rightarrow 0\;,
\end{split}
\end{equation*}
as $n$ goes to infinity. Markov's inequality concludes the proof.
\subsection{Lower bound}
We shall assume that $\esssup W<1$. Let us justify this step.
Suppose that $W$ is an arbitrary graphon. We can then take a sequence of graphons $W_1,W_2,\ldots$, where $W_j=\min(W,1-\frac1j)$ (pointwise).\Referee{17} Then~\eqref{eq:kappa2} tells us that $\kappa(W_j)\rightarrow \kappa(W)$ (even in the case $\kappa(W)=+\infty$). Thus, it suffices to prove a lower bound for each of the graphons $W_j$.

Let $\epsilon>0$ be arbitrary. We apply Lemma~\ref{lem:zoom} to find a set $A\subset \Omega$ of positive measure such that for the subgraphon $U=W\restriction_{A\times A}$ we have $\frac{1}{\xi(U)}\ge\kappa(W)-\epsilon$. Lemma~\ref{lem:firstandsecond} then tells us that asymptotically almost surely, $\omega(\RG(n,U))\ge (\kappa(W)-2\epsilon)\log n$. Since there is a coupling of $G=\RG(n,W)$ and $G'=\RG(\frac{\lambda(A)n}2,U)$ such that $G$ asymptotically almost surely contains a copy of $G'$, we obtain that (cf.~\eqref{eq:PsiSubgraphon}),
$$\omega(\RG(n,W))\ge (\kappa(W)-3\epsilon)\log n \text{ asymptotically almost surely.}$$ Since $\epsilon>0$ was arbitrary, this completes the proof of Theorem~\ref{thm:main}.

\section{Concluding remarks}\label{sec:concluding}
Our concluding remarks concern possibilities of extending the main result, Theorem~\ref{thm:main}.
\subsection{Sharpening the results} As mentioned in Section~\ref{sec:Intro}, Matula, Grimmett and McDiarmid proved for $p\in(0,1)$ an asymptotic concentration of $\omega(\RG(n,p))$ on two consecutive values for which they provided an explicit formula. It is possible that when, say, $0<\essinf W\le \esssup W<1$, then $\omega(\RG(n,W))$ is asymptotically concentrated on two consecutive values.

\subsection{Sparse inhomogeneous random graphs}\label{sec:sparse}
Let us look at our set of problems for $\RG(n,p_n\cdot W)$, where $p_n\rightarrow 0$, i.e., the model introduced in Section~\ref{ssec:relatedliterature}. Note that Remark~\ref{rem:CliqueIndep} is no longer valid: the problem of maximum clique and maximum independent set in $\RG(n,p_n\cdot W)$ is genuinely different. It turns out that the more interesting problem is that of the independent set. For the Erd\H{o}s--R\'enyi random graph $\RG(n,p_n)$, the problem of determining the independence number is essentially solved by the above mentioned work \cite{Matula:LargestClique,GriMcD:ColouringRandom}, and by the work of Frieze~\cite{Frieze:Independence} down to the range $p_n\gg \frac{1}{n}$. Note that the regime $p_n\ll \frac{1}{\sqrt n}$ is more subtle as the second moment argument does not work, and indeed Frieze's contribution was in establishing concentration of the count of large independent sets by alternative means. The regime $p_n=C/n$ seems to require methods from statistical physics. In the related model of random regular graphs, these methods have already provided an answer~\cite{DiSlSu:MaxIndRandReg}. 

It would be of interest to see whether the methods we developed in this paper can give an answer also for the independence number in sparser inhomogeneous random graphs. It seems that the two core ingredients of our proof, Lemma~\ref{lem:secondsamplinglemma} and the second moment argument do have sparse counterparts.
\begin{itemize}
\item  The sparse counterpart to Lemma~\ref{lem:secondsamplinglemma} is \cite[Theorem 2.14]{BHHZ:LPgraphlimits} which says that if $p_n\gg \frac 1n$ then the sequence $\frac1{p_n}\cdot\mathbb G(n,p_n\cdot W)$ (here, the factor $\frac1{p_n}$ in front of the random graph $\mathbb G(n,p_n\cdot W)$ denotes edge weighting; this is the natural way to deal with the scaling in this situation) converges to $W$ in the cut-distance almost surely.
\item Our second moment argument is complicated but it builds on the seminal work~\cite{Matula:LargestClique,GriMcD:ColouringRandom} which works down to the range $p_n=\Theta(\frac{1}{\sqrt n})$. Thus, at least when $W\in L^\infty(\Omega^2)$, our methods possibly extend to this range. The situation when $W\in L^p(\Omega^2)$ for some general $p$ is probably more subtle.

Of course, one might ask whether the methods used by Frieze~\cite{Frieze:Independence} could be extended to the inhomogeneous setting, thus possibly giving results even for $\frac{1}{n}\ll p_n< \frac{1}{\sqrt n}$. This however goes beyond the scope of this paper.
\end{itemize}
\section*{Acknowledgments}
JH thanks Lutz Warnke for suggesting the idea of the proof of Theorem~\ref{thm:concentration}. We thank the referees for their helpful comments.

\bigskip
The contents of this publication reflects only the authors' views and not necessarily the views of the European Commission of the European Union.
\bibliographystyle{plain} 
\bibliography{bibl}

\begin{thebibliography}{10}

\bibitem{BolBorChaRio:PercolationDense}
B.~Bollob{\'a}s, Ch. Borgs, J.~Chayes, and O.~Riordan.
\newblock Percolation on dense graph sequences.
\newblock {\em Ann. Probab.}, 38(1):150--183, 2010.

\bibitem{BolJanRio:PhaseTransition}
B.~Bollob{\'a}s, S.~Janson, and O.~Riordan.
\newblock The phase transition in inhomogeneous random graphs.
\newblock {\em Random Structures Algorithms}, 31(1):3--122, 2007.

\bibitem{BHHZ:LPgraphlimits}
C.~Borgs, J.~T. Chayes, H.~Cohn, and Y.~Zhao.
\newblock An {$L^p$} theory of sparse graph convergence {I}: limits, sparse
  random graph models, and power law distributions, 2014.
\newblock arXiv:1401.2906.

\bibitem{Borgs2008c}
C.~Borgs, J.~T. Chayes, L.~Lov{\'a}sz, V.~T. S{\'o}s, and K.~Vesztergombi.
\newblock Convergent sequences of dense graphs. {I}. {S}ubgraph frequencies,
  metric properties and testing.
\newblock {\em Adv. Math.}, 219(6):1801--1851, 2008.

\bibitem{CoFoSu:Sidorenko}
D.~Conlon, J.~Fox, and B.~Sudakov.
\newblock An approximate version of {S}idorenko's conjecture.
\newblock {\em Geom. Funct. Anal.}, 20(6):1354--1366, 2010.

\bibitem{DevFra}
L.~Devroye and N.~Fraiman.
\newblock Connectivity of inhomogeneous random graphs.
\newblock {\em Random Structures Algorithms}, 45(3):408--420, 2014.

\bibitem{DiSlSu:MaxIndRandReg}
J.~Ding, A.~Sly, and N.~Sun.
\newblock Maximum independent sets on random regular graphs, 2013.

\bibitem{FHHMZ}
M.~Fabian, P.~Habala, P.~H{\'a}jek, V.~Montesinos, and V.~Zizler.
\newblock {\em Banach space theory}.
\newblock CMS Books in Mathematics/Ouvrages de Math\'ematiques de la SMC.
  Springer, New York, 2011.
\newblock The basis for linear and nonlinear analysis.

\bibitem{FrMi:Diameter}
N.~Fraiman and D.~Mitsche.
\newblock The diameter of inhomogeneous random graphs, 2015.
\newblock arXiv:1510.08882.

\bibitem{Frieze:Independence}
A.~M. Frieze.
\newblock On the independence number of random graphs.
\newblock {\em Discrete Math.}, 81(2):171--175, 1990.

\bibitem{GriMcD:ColouringRandom}
G.~R. Grimmett and C.~J.~H. McDiarmid.
\newblock On colouring random graphs.
\newblock {\em Math. Proc. Cambridge Philos. Soc.}, 77:313--324, 1975.

\bibitem{Hatami:Sidorenko}
H.~Hatami.
\newblock Graph norms and {S}idorenko's conjecture.
\newblock {\em Israel J. Math.}, 175:125--150, 2010.

\bibitem{HolLasLei:StochasticBlockmodels}
P.~W. Holland, K.~B. Laskey, and S.~Leinhardt.
\newblock Stochastic blockmodels: first steps.
\newblock {\em Social Networks}, 5(2):109--137, 1983.

\bibitem{KaKoPa:PhaseTransitionBlock}
M.~Kang, Ch. Koch, and A.~Pach{\'o}n.
\newblock The phase transition in multitype binomial random graphs.
\newblock {\em SIAM J. Discrete Math.}, 29(2):1042--1064, 2015.

\bibitem{Kechris}
A.~S. Kechris.
\newblock {\em Classical descriptive set theory}, volume 156 of {\em Graduate
  Texts in Mathematics}.
\newblock Springer-Verlag, New York, 1995.

\bibitem{LiSze:Sidorenko}
J.~L.~X. Li and B.~Szegedy.
\newblock On the logarithimic calculus and {S}idorenko's conjecture.
\newblock Preprint (arXiv:1107.1153).

\bibitem{Lovasz:Sidorenko}
L.~Lov{\'a}sz.
\newblock Subgraph densities in signed graphons and the local
  {S}imonovits-{S}idorenko conjecture.
\newblock {\em Electron. J. Combin.}, 18(1):Paper 127, 21, 2011.

\bibitem{Lovasz:LimitBook}
L.~Lov{\'a}sz.
\newblock {\em Large networks and graph limits}, volume~60 of {\em American
  Mathematical Society Colloquium Publications}.
\newblock American Mathematical Society, Providence, RI, 2012.

\bibitem{LovSze07}
L.~Lov\'asz and B.~Szegedy.
\newblock Szemer\'edi's {L}emma for the analyst.
\newblock {\em J. Geom. and Func. Anal}, 17:252--270, 2007.

\bibitem{LoSz:RegularityTopology}
L.~Lov{\'a}sz and B.~Szegedy.
\newblock Regularity partitions and the topology of graphons.
\newblock In {\em An irregular mind}, volume~21 of {\em Bolyai Soc. Math.
  Stud.}, pages 415--446. J\'anos Bolyai Math. Soc., Budapest, 2010.

\bibitem{Matula:LargestClique}
D.~W. Matula.
\newblock The largest clique size in a random graph.
\newblock Technical report, Department of Computer Science, Southern Methodist
  University, 1976.

\bibitem{MolloyReed}
M.~Molloy and B.~Reed.
\newblock {\em Graph colouring and the probabilistic method}, volume~23 of {\em
  Algorithms and Combinatorics}.
\newblock Springer-Verlag, Berlin, 2002.

\bibitem{Rudin}
W.~Rudin.
\newblock {\em Real and complex analysis}.
\newblock McGraw-Hill Book Co., New York, third edition, 1987.

\bibitem{Sidorenko:Conjecture}
A.~Sidorenko.
\newblock A correlation inequality for bipartite graphs.
\newblock {\em Graphs Combin.}, 9(2):201--204, 1993.

\bibitem{Simonovits:Sidorenko}
M.~Simonovits.
\newblock Extremal graph problems, degenerate extremal problems, and
  supersaturated graphs.
\newblock In {\em Progress in graph theory ({W}aterloo, {O}nt., 1982)}, pages
  419--437. Academic Press, Toronto, ON, 1984.

\bibitem{Tur:LargestSubcritical}
T.~S. Turova.
\newblock The largest component in subcritical inhomogeneous random graphs.
\newblock {\em Combin. Probab. Comput.}, 20(1):131--154, 2011.

\bibitem{vanderHofstad:Critical}
R.~van~der Hofstad.
\newblock Critical behavior in inhomogeneous random graphs.
\newblock {\em Random Structures Algorithms}, 42(4):480--508, 2013.

\end{thebibliography}

\appendix
\section{A simplified version of Lemma~\ref{lem:maxl1}}
Here we provide a weaker version of Lemma~\ref{lem:maxl1} which deals with the case that the supremum in~\eqref{eq:defKappa} is attained. While this version is not sufficient for our purposes we decided to offer it to the reader because its proof is based on the same idea yet is stripped off technicalities.
\begin{lemma}\label{lem:maxl1App}
	Suppose that $W$ is an arbitrary graphon, and suppose that $f^*$ is an admissible histogram for $W$ for which $\|f^*\|_1=\kappa(W)$. Then every subhistogram of $f^*$ is admissible for $W$. 
\end{lemma}
\begin{proof}
	The statement follows immediately from Claim~\ref{cl:AndrasClaimApp} (which is a simplified version of Lemma~\ref{cl:AndrasClaim}). 
	
	We abbreviate $\Gamma(\cdot, W)$ as $\Gamma(\cdot)$. Also, when we say ``admissible'', we mean with respect to $W$.
	\begin{claim}\label{cl:AndrasClaimApp}
		Assume that $g$ is an arbitrary admissible histogram. Suppose further that $g=g'+g''$ for some histograms $g'$ and $g''$. Then either $g'$ is admissible, or there exist $\epsilon_1,\epsilon_2\in(0,1)$ such that for $g^*=(1-\epsilon_1)g'+(1+\epsilon_2)g''$ we have that $g^*$ is admissible, and $\|g^*\|_1>\|g\|_1$.
	\end{claim}
	\begin{proof}[Proof of Claim~\ref{cl:AndrasClaimApp}]
		For $\epsilon_1,\epsilon_2\in(0,1)$, let us write $g^*(\eps_1,\eps_2)=(1-\epsilon_1)g'+(1+\epsilon_2)g''$. We define numbers $A$, $B$, $C$, $D$, and $E$ as in~\eqref{eq:ABCDE}.
		Note that $A,B,C,D,E\geq 0$. For any $\epsilon_1,\epsilon_2\in(0,1)$, the difference $\Gamma(g^*(\eps_1,\eps_2))-\Gamma(g)$ can be expressed as
		$$(1-\eps_1)A + (1+\eps_2)B - (1-\eps_1)^2C - (1+\eps_2)^2D - (1-\eps_1)(1+\eps_2)E - \left( A + B - C - D - E \right)$$
		$$= \eps_1 ( -A + 2C + E ) + \eps_2( B - 2D -E ) - \eps_1^2C - \eps_2^2D + \eps_1\eps_2E.$$
		In particular, if $\eps_2=\frac{A}{B}\eps_1$ then we have
		\begin{equation}
		\label{choice of eps2App}
		\Gamma\left(g^*\left(\eps_1,\frac{A}{B}\eps_1\right)\right)-\Gamma(g) = \eps_1\left(2C+E-\frac{2AD}{B}-\frac{AE}{B}\right) + \eps_1^2\left(-C-\frac{A^2D}{B^2}+\frac{AE}{B}\right).
		\end{equation}
		Now let us assume that $\Gamma(g')<0$, i.e. $A<C$. Then we have
		\begin{equation}
		\label{the linear coefApp}
		\begin{split}
		2C+E-\frac{2AD}{B}-\frac{AE}{B}&>2A-\frac{2AD}{B}-\frac{AE}{B}\\
		&>2\frac{A}{B} ( B - D - E)\\
		&>2\frac{A}{B}\left(B-D-E+(A-C)\right)\\
		&=2\frac{A}{B}\Gamma(g)\\
		&\ge 0.
		\end{split}
		\end{equation}
		By (\ref{the linear coefApp}), the right hand side of (\ref{choice of eps2App}) is a quadratic expression (in the variable $\eps_1$) with a positive linear coeficient.\Referee{8} Therefore there is $\eps_1>0$ (which we fix now) small enough such that $\epsilon_1,\frac{A}{B}\epsilon_1\in(0,1)$ and
		\begin{equation}
		\label{unif positiveApp}
		\Gamma\left(g^*\left(\eps_1,\frac{A}{B}\eps_1\right)\right)>\Gamma(g)\left(\geq 0\right).
		\end{equation}
		Since the function $(\eps_1,\eps_2)\mapsto \Gamma\left(g^*\left(\eps_1,\eps_2\right)\right)$ is obviously continuous, we can find $\eps_2\in\left(\frac{A}{B}\eps_1,1\right)$ such that $\Gamma\left(g^*\left(\eps_1,\eps_2\right)\right)$ is still nonnegative. Then we have
		\begin{equation*}
		\|g^*\left(\eps_1,\eps_2\right)\|_1 = (1-\eps_1)A + (1+\eps_2)B
		>(1-\eps_1)A + \left(1+\frac{A}{B}\eps_1\right)B = A+B = \|g\|_1.
		\end{equation*}
		This finishes the proof.
	\end{proof}
\end{proof}

\end{document}